\documentclass[12pt]{amsart}

\usepackage{sfilip_presets}
\usepackage{amsmath}
\usepackage{amsthm}
\usepackage{booktabs}

\usepackage{longtable}

\newif \iffig
\figtrue

\iffig
  \usepackage{wrapfig,graphicx}
  \numberwithin{figure}{subsection}
\fi

\usepackage{tikz-cd}

\usepackage[stretch=10, final]{microtype}
\usepackage{lmodern}
\usepackage[T1]{fontenc}

\usepackage{etoolbox}
\usepackage{needspace}
\AtBeginEnvironment{theorem}{\Needspace{5\baselineskip}}
\usepackage{titletoc}

\let\oldtocsection=\tocsection
\let\oldtocsubsection=\tocsubsection
\let\oldtocsubsubsection=\tocsubsubsection

\renewcommand{\tocsection}[2]{\hspace{0em}\oldtocsection{#1}{#2}}
\renewcommand{\tocsubsection}[2]{\hspace{1.75em}\oldtocsubsection{#1}{#2}}
\renewcommand{\tocsubsubsection}[2]{\hspace{2em}\oldtocsubsubsection{#1}{#2}}

\makeatletter
\renewcommand{\@secnumfont}{\bfseries}
\makeatother
\makeatletter
\renewcommand{\section}{
\@startsection {section}
				   {1}
				   {\z@}%
                                   {-3.5ex \@plus -1ex \@minus -.2ex}%
                                   {2.3ex \@plus.2ex}%
                                   {\centering \normalfont \Large \bfseries \sffamily}}%
\makeatother

\makeatletter
\renewcommand{\subsection}
{\@startsection{subsection}
				     {2}
				     {\z@}%
                                     {-3.25ex\@plus -1ex \@minus -.2ex}%
                                     {1.5ex \@plus .2ex}%
                                     {\normalfont \large \bfseries \sffamily}}% from \large
\makeatother

\makeatletter
\renewcommand{\subsubsection}
{\@startsection{subsubsection}
				     {3}
				     {\z@}%
                                     {3.25ex \@plus1ex \@minus.2ex}%
				     {-\fontdimen 2\font }
                                     {\normalfont\normalsize\bfseries}}
\makeatother

\makeatletter
\renewcommand{\paragraph}{%
\@startsection {paragraph}
				    {4}
				    {\z@}
				    {\z@}
				    {-\fontdimen 2\font }
				    {\normalfont\normalsize\bfseries}
}
\makeatother

\newif\ifdebug
\debugfalse

\ifdebug
  \usepackage[status=draft,author=]{fixme}
  \fxsetup{inline, marginclue,theme=color}
  \usepackage{showkeys}
\fi
\usepackage{hyperref}

\usepackage{color}
\definecolor{darkred}{rgb}{0.4,0,0}
\definecolor{darkgreen}{rgb}{0,0.5,0}
\definecolor{darkblue}{rgb}{0,0,0.4}
\hypersetup{
    pdftitle={Counting special Lagrangian fibrations in twistor families of K3 surfaces},	% title
    pdfauthor={Simion Filip},	% author
    pdfsubject={MSC subject classification: primary 14J28; secondary 22F30, 53C26},		% subject of the document
    pdfkeywords={sLag, counting, K3 surfaces},	% list of keywords
    pdfnewwindow=true,		% links in new window
    colorlinks=true,		% false: boxed links; true: colored links
    linkcolor=darkblue,		% color of internal links (change box color with linkbordercolor)
    citecolor=darkred,		% color of links to bibliography
    filecolor=darkblue,		% color of file links
    urlcolor=darkblue,		% color of external links
    pdfborder={0 0 0},
    breaklinks=true
}
\def\subsek~{\S{}}

\def\subsubsek~{\S{}}

\def\equationautorefname~#1\null{%
  Eqn.~(#1)\null
}
\newtheoremstyle{mytheoremstyle}% name
    {5pt}	                % Space above
    {5pt}                    	% Space below
    {\itshape}                  % Body font
    {}                          % Indent amount
    {\bfseries}                 % Theorem head font
    {.}                          % Punctuation after theorem head
    {.5em}                      % Space after theorem head
    {}  			% Theorem head spec (can be left empty, meaning ‘normal’)

\newtheoremstyle{mydefinitionstyle}% name
    {5pt}	                % Space above
    {5pt}                    	% Space below
    {}                  	% Body font
    {}                          % Indent amount
    {\bfseries}                 % Theorem head font
    {.}                          % Punctuation after theorem head
    {.5em}                      % Space after theorem head
    {}  			% Theorem head spec (can be left empty, meaning ‘normal’)

\theoremstyle{mytheoremstyle}			% Uses the custom theorem style

\newtheorem{theoremmain}{Theorem}

\swapnumbers			% changes numbering to "2.3.1 Theorem" rather than "Theorem 2.3.1"

\newtheorem{theorem}{Theorem}[subsection]	% Theorem numbers are within subsections

\numberwithin{equation}{subsection}		% Equations are numbered within subsections

\makeatletter					% Makes theorems numbered with the same counter as equations
\let\c@theorem \c@equation
\makeatother

\makeatletter					% Makes subsubsections numbered with the same counter as equations
    \let\c@subsubsection\c@equation
\makeatother

\makeatletter					% Makes figures numbered with the same counter as equations
    \let\c@figure\c@equation
\makeatother
\usepackage{aliascnt}

\newaliascnt{lemma}{theorem}
\newtheorem{lemma}[lemma]{Lemma}
\aliascntresetthe{lemma}

\newaliascnt{proposition}{theorem}
\newtheorem{proposition}[proposition]{Proposition}
\aliascntresetthe{proposition}

\newaliascnt{corollary}{theorem}

\aliascntresetthe{corollary}

\theoremstyle{mydefinitionstyle}

\newaliascnt{exercise}{theorem}

\aliascntresetthe{exercise}

\newaliascnt{definition}{theorem}
\newtheorem{definition}[definition]{Definition}
\aliascntresetthe{definition}

\newaliascnt{remark}{theorem}
\newtheorem{remark}[remark]{Remark}
\aliascntresetthe{remark}

\newaliascnt{example}{theorem}

\aliascntresetthe{example}

\newaliascnt{question}{theorem}

\aliascntresetthe{question}

\usepackage{etoolbox}
\makeatletter
\patchcmd{\@maketitle}
  {\ifx\@empty\@dedicatory}
  {\ifx\@empty\@date \else {\vskip3ex \centering\footnotesize\@date\par\vskip1ex}\fi
   \ifx\@empty\@dedicatory}
  {}{}
\patchcmd{\@adminfootnotes}
  {\ifx\@empty\@date\else \@footnotetext{\@setdate}\fi}
  {}{}{}
\makeatother
\newcommand{\Lip}{\operatorname{Lip}}

\DeclareMathOperator{\inj}{inj}

\begin{document}

%====================================================
%		Title matters
%	also: Add it in the hyperref setup
%
\title[Counting sLag fibrations in twistor families of K3s]{Counting special Lagrangian fibrations in twistor families of K3 surfaces}
%
%	change to \title[short title]{Actual title}
%	if what appears on headers is too long
%====================================================
%		Date of revision
%   Date of revision
\thanks{
\noindent Revised \textsc{\today}.
Mots cl\'es: Surfaces K3, Fibrations Lagrangiennes, \'{E}quidistribution}

%-----------Date First posted---------------------
\date{
with an Appendix by \textsc{Nicolas Bergeron} and \textsc{Carlos Matheus}\\
at \href{https://arxiv.org/abs/1703.01746}{arXiv:1703.01746} \\
\quad \\
September 2016}

%====================================================
%   Author info
\author{
Simion Filip
}
%   Address
\address{
\parbox{\textwidth}{
  Simion Filip\\
  School of Mathematics, Institute for Advanced Study\\
  1 Einstein Drive, Princeton, NJ 08540, USA}
  }
\email{{sfilip@math.ias.edu}}

% \address{
% \parbox{\textwidth}{
% Nicolas Bergeron\\
% Universit\'e Pierre et Marie Curie, Institut de Math\'ematiques de Jussieu, CNRS (UMR 7586), 4, place Jussieu 75252\\
% Paris Cedex 05, France}
% }
% \email{nicolas.bergeron@imj-prg.fr.}
% % \author{Carlos Matheus}
% \address{
% \parbox{\textwidth}{
%   Carlos Matheus\\
%   CMLS, \'{E}cole Polytechnique\\
%   CNRS (UMR 7640), 91128 Palaiseau, France
% }
% }
% \email{carlos.matheus@math.cnrs.fr}
%
%====================================================
%
%		Abstract
%
\begin{abstract}
The number of closed billiard trajectories in a rational-angled polygon grows quadratically in the length.
This paper gives an analogue on K3 surfaces, by considering special Lagrangian tori.
The analogue of the angle of a billiard trajectory is a point on a twistor sphere, and the number of directions admitting a special Lagrangian torus fibration with volume bounded by $ V $ grows like $ V^{20} $ with a power-saving term.
Bergeron--Matheus have explicitly estimated the exponent of the error term as $ {20-\frac{4}{697633} }$.
The counting result on K3 surfaces is deduced from a count of primitive isotropic vectors in indefinite lattices, which is in turn deduced from equidistribution results in homogeneous dynamics.

\noindent \hrulefill

\begin{center}
	\textbf{Comptage des fibrations en Lagrangiens sp\'eciaux dans les familles de twisteur des surfaces K3}
\end{center}

Le nombre de trajectoires ferm\'{e}es des billards dans un polygone \`{a} angles rationnels a une croissance quadratique comme fonction de la longuer.
Cet article donne un analogue sur les surfaces K3, en consid\'{e}rant des tores Lagrangiens sp\'{e}ciaux.
L'analogue de l'angle d'une trajectoire de billard est un point sur une sph\`{e}re de twisteur, et le nombre de directions admettant une fibration en lagrangiens sp\'{e}ciaux avec un volume born\'{e} par $ V $ croit comme $ V ^ {20} $ avec un terme d'erreur.
Bergeron--Matheus ont explicitement estim\'{e} l'exposant du terme d'erreur \`{a} $ {20- \frac {4} {697633}} $.
Le comptage sur les surfaces K3 est d\'eduit \`a partir d'un comptage de vecteurs isotropes primitifs dans des r\'eseaux ind\'efinis, qui est \`a sont tour d\'eduit des r\'esultats d'\'equidistribution en dynamique homog\`ene.
\end{abstract}
%====================================================

%====================================================
%
%		Title, TOC, FixMe
%
\maketitle
%
%		TOC
% \noindent\hrulefill
{\tableofcontents}
% \nointerlineskip
% \noindent \hrulefill
%====================================================

%		Things to be fixed
\ifdebug
  \listoffixmes
\fi
%====================================================

\section{Introduction}

\subsection{Motivations}

\paragraph{Billiards}
Consider a regular $ n $-gon and billiard trajectories in it.
Veech \cite{Veech_Eisenstein} proved that the number of \emph{closed} billiard trajectories of length at most $ L $ is asymptotic to $ c_n L^2 $ for an explicit constant $ c_n $.
For general rational-angled polygons Masur \cite{Masur_lower} proved that the number of closed trajectories has quadratic upper and lower bounds, and results of Eskin, Mirzakhani, and Mohammadi \cite{EMM} imply a quadratic asymptotic in an averaged sense.
For a general polygon, it is not known if a single closed trajectory exists.

\paragraph{Translation surfaces}
These results are proved by studying the moduli space of ``flat'' or ``translation'' surfaces, i.e. Riemann surfaces $ X $ equipped with a holomorphic $ 1 $-form $ \Omega $.
These carry a flat metric $ \frac{\sqrt{-1}}{2}\Omega\wedge \conj{\Omega} $ with singularities at the zeros of $ \Omega $.
To go from a rational-angled polygon to a surface, the polygon is unfolded by reflections in the sides to finitely many copies in $ \bR^2\cong \bC $ and the sides are glued by translations, so that the $ 1 $-form $ \Omega:=dz $ on $ \bC $ descends to the glued surface.

\paragraph{Straight lines}
Of course, closed geodesics are just length-minimizing curves.
But on a translation surface $ (X,\Omega) $ a horizontal curve can also be described as one for which the imaginary part $ \Im \Omega $ restricts to zero.
Similarly, straight lines at angle $ \theta $ correspond to those for which $ \Im\left(e^{\sqrt{-1}\theta}\Omega\right) $ restricts to zero.

Another feature of closed geodesics on a translation surface is that they occur in families: a small parallel deformation will again close up.
Moreover, closed geodesics occur in a dense (on the unit circle) set of directions.

\paragraph{K3 surfaces}
We consider an extension of the above constructions to K3 surfaces -- compact complex $ 2 $-dimensional manifolds which admit a nowhere vanishing holomorphic $ 2 $-form $ \Omega $ and are simply connected.
By Yau's solution of the Calabi conjecture \cite{Yau_Ricci} on a K3 any \Kahler cohomology class has a unique Ricci-flat \Kahler representative $ \omega $.

\paragraph{Special Lagrangians}
One generalization of straight lines on a translation surface to higher dimensions are special Lagrangian manifolds.
Namely, on a complex $ n $-manifold $ X $ with a holomorphic $ n $-form $ \Omega $ and Ricci-flat \Kahler metric $ \omega $, a submanifold $ L\subset X $ is \emph{special Lagrangian} if:
\begin{enumerate}
	\item $ L $ is Lagrangian for the symplectic form $ \omega $ (i.e. $ \dim_{\bR} L =n $ and $ \omega\vert_L \equiv 0 $).
	\item The restriction $ \Omega\vert_L $ gives the Riemannian volume form on $ L $ (in particular $ \Im \Omega\vert_L\equiv 0 $).
\end{enumerate}
Note that special Lagrangians are automatically minimal surfaces, i.e. locally volume minimizing (because they are calibrated manifolds).

\paragraph{Fibrations}
We'll regard special Lagrangian tori as an analogue on K3 surfaces of closed billiard trajectories.
When a single such torus exists on a K3 surface $ X $, it deforms to give a special Lagrangian fibration $ X \to B $ with finitely many singular fibers.
Fibrations $X\to B$ will always be assumed to have connected fibers and two fibrations will be regarded as equivalent if there is a map on the bases $B\to B'$ such that the corresponding diagram commutes.
% This induces an integral-affine structure on the base and is related to the Strominger--Yau--Zaslow \cite{SYZ} picture of mirror symmetry (see \autoref{sssec:Zaff_struct} for a discussion). FIX OR REMOVE THIS.

\paragraph{Twistor families}
A Ricci-flat K3 surface $ (X,\Omega,\omega) $ naturally sits in a family $ \cX\to \bS^2 $ in which all the fibers are isometric as Riemannian manifolds, but the complex structure changes.
The points on the base $ \bS^2 $ of such a twistor family will serve as directions in which we will look for special Lagrangian fibrations.

\paragraph{Special Lagrangian $ \leftrightarrow $ holomorphic correspondence}
If a fiber $ (\cX_{t_0}, \Omega_{t_0}, \omega_{t_0}) $ of a twistor family admits a special Lagrangian fibration for some $ t_0\in \bS^2 $, then the same fibration is special Lagrangian for $ t $ in an entire equator containing $ t_0 $.
At the poles corresponding to that equator, the same special Lagrangian fibration becomes an \emph{elliptic} fibration, i.e. the fibers are complex elliptic curves for the corresponding complex structure (see \autoref{prop:hol_slag_corresp}).
In both types of fibrations the Riemannian volumes of each fiber agree and give a notion of volume of a fibration.

\subsection{Results}
\label{ssec:results}

The geometry of a K3 surface is directly related to the Hodge structure on its middle cohomology group.
Recall that the integral cohomology of a K3 surface equipped with cup product is isomorphic to the even unimodular lattice $\bI_{3,19}$ of signature $(3,19)$.
This lattice is unique up to isomorphism and will be denoted by $\Lambda_\bZ$ and extensions of scalars by a ring $R$ will be denoted $\Lambda_R$.

The information required to define a twistor family is encoded in a positive-definite $3$-plane $P\subset \Lambda_\bR$ which defines a semi-norm
\begin{align*}
  \norm{e}_P := \sup_{\substack{\kappa \in P ,\, \kappa^2=1}} \kappa \cdot e .
\end{align*}
Moreover, using the decomposition $\Lambda_\bR = P\oplus P^\perp $ we get a decomposition $e=e_P\oplus e_{P^\perp}$ and we clearly have $e_P\cdot e_P = \norm{e}_P^2$.

In order to perform a count of special Lagrangian fibrations we first establish the following counting result in the homogeneous setting (see \autoref{thm:homogeneous_counting} for a precise statement).

\begin{theoremmain}
  \label{thm:main_homogeneous}
  The number $N(V)$ of primitive integral isotropic\footnote{That is: $e\in \Lambda_\bZ$, $e^2=0$, and $e\neq k\cdot e'$ with $k\in \bN$, $e'\in \Lambda_\bZ$} vectors $e$ with $\norm{e}_P\leq V$ satisfies:
  \[
    N(V) = C\cdot V^{20} + O(V^{20-\delta})
  \]
  for a universal constant $C$, independent of the plane $P$.

  Moreover, the unit vectors $\frac 1 {\norm{e}_P} e_P$ quantitatively equidistribute on the unit sphere of $P$.

  \noindent \textbf{Bergeron--Matheus} \cite{Bergeron_Matheus} The above asymptotic holds for any constant $ \delta $ less than $ \frac   4{697633} \approx 6\cdot 10^{-6} $.
\end{theoremmain}

The leading constant $ C>0 $ that appears in the theorem is the ratio of volumes of two homogeneous spaces.
It is computed in \autoref{sssec:tamagawa} up to rational factors to be $C\in (\pi^{20} \zeta(11))^{-1}\cdot \bQ$.

The following counting result for K3 surfaces then follows by combining classical results on the geometry of K3 surfaces with \autoref{thm:main_homogeneous}.

\begin{theoremmain}
  \label{thm:main_elliptic_counting}
  Let $ \cX \to \bS^2 $ be a generic (in the sense of \autoref{def:generic_twistor}) twistor family of K3 surfaces.
  Let $ N(V) $ be the number of elliptic fibrations of volume at most $ V $ which occur in some member of the twistor family.
  Then we have the asymptotic
  \begin{align*}
  N(V) = C\cdot V^{20} + O(V^{20-\delta})
  \end{align*}
  for some universal constants $ C, \delta >0 $.
  Moreover, the collection of points at which the fibrations occur are equidistributed on the twistor sphere $ \bS^2 $.
\end{theoremmain}

By the equivalence between special Lagrangian fibrations and elliptic fibrations, \autoref{thm:main_elliptic_counting} implies the following one.

\begin{theoremmain}
	\label{thm:main_slag_counting}
	Let $ \cX \to \bS^2 $ be a generic (in the sense of \autoref{def:generic_twistor}) twistor family of K3 surfaces.
	Let $ N(V) $ be the number of special Lagrangian fibrations of volume at most $ V $ which occur in some member of the twistor family.
	Then we have the asymptotic
	\begin{align*}
	N(V) = C\cdot V^{20} + O(V^{20-\delta})
	\end{align*}
	for some universal constants $ C, \delta >0 $.
	Moreover, the collection of equators at which the fibrations occur are equidistributed on the twistor sphere $ \bS^2 $.
\end{theoremmain}

Assuming \autoref{thm:main_homogeneous} and some classical facts about the geometry of K3 surfaces, \autoref{thm:main_elliptic_counting} and \autoref{thm:main_slag_counting} are established below.

\paragraph{Genericity}
	The assumption in \autoref{thm:main_elliptic_counting} and \autoref{thm:main_slag_counting} that the twistor family is generic can be made explicit and one can present concrete twistor families to which it applies (see \autoref{prop:generic}).
  	
	In terms of the positive-definite $3$-plane $P\subset \Lambda_\bR$ associated to the twistor family, the condition is that for any $2$-plane $S\subset P$ with orthogonal (inside $\Lambda_\bR$) subspace $S^\perp$, we have:
  	if $\exists v \in (S^\perp\cap \Lambda_\bZ) $ with $v^2=-2$ then $\operatorname{rank} (S^\perp\cap \Lambda_\bZ)\leq 1$.
  The group $S^\perp \cap \Lambda_\bZ$ coincides with the N\'eron--Severi group of a corresponding K3 surface in the twistor family.
	It would be interesting to extend the counting result to more general twistor families, with the appropriate change of constant and exponent in the statement.

\paragraph{Riemannian version} % (fold)
\label{par:riemannian_version}
\autoref{thm:main_elliptic_counting} and \autoref{thm:main_slag_counting} have a purely Riemannian interpretation.
Namely, fix a Ricci-flat K3 surface and assume that the associated twistor family is generic.
Then, forgetting the complex structure and considering the Riemannian manifold $(X,g)$ we have:

\begin{theoremmain}
  \label{thm:main_riemannian}
  Let $N(V)$ be the number of connected, absolutely volume-minimizing surfaces $T\to X$ with $\Vol(T)\leq V$ and vanishing self-intersection (i.e. $[T]^2=0$), considered equivalent if they are in the same homology class.
  Then we have the asymptotic
    \begin{align*}
    N(V) = C\cdot V^{20} + O(V^{20-\delta})
    \end{align*}
    for universal constants $C,\delta>0$.
\end{theoremmain}
Above, a surface $T\to X$ means a smooth map from a connected surface $T$ to $X$.
It induces a fundamental class $[T]\in H_2(X)$ and has a volume by pullback of the Riemannian metric.
An absolutely volume-minimizing surface is one whose volume is less than or equal to that of any other surface in the same homology class.

% paragraph riemannian_version (end)

% \paragraph{Sketch of the counting problem}
% The second cohomology of a K3 surface is isomorphic to $ \bI_{3,19} $ -- the unique even, unimodular lattice of signature $ (3,19) $.
% The data of a twistor family gives a real positive definite $ 3 $-plane $ P\subset \bI_{3,19}(\bR) $ which induces a semi-norm
% We are interested in counting primitive integral null vectors $ e\in \bI_{3,19} $ such that $ \norm{e}_{P}\leq V $.
% This problem can be approached using dynamics on homogeneous spaces.

\subsection{Proofs of \autoref{thm:main_elliptic_counting}, \autoref{thm:main_slag_counting}, and \autoref{thm:main_riemannian}}

We now present the proofs of \autoref{thm:main_elliptic_counting}, \autoref{thm:main_slag_counting}, and \autoref{thm:main_riemannian} assuming \autoref{thm:main_homogeneous} and some facts which are explained in \autoref{sec:K3_background}.

Fix a K3 surface $X$ and an isomorphism $H^2(X)\cong \Lambda$.
Pick a \Kahler cohomology class with representative a Ricci-flat metric $\omega$, such that the real, positive-definite $3$-plane
\[
	P = \left( (H^{2,0} \oplus H^{0,2}) \cap H^2(X;\bR)\right) \oplus \bR [\omega] \subset \Lambda_\bR
\]
is generic in the sense of \autoref{def:generic_twistor}.

The unit sphere $\bS^2(P)$ parametrizes the points in the twistor family, as explained in \autoref{sssec:twistor_families}.
Each point $u \in \bS^2(P)$ corresponds to a new complex structure $J_u$ on $X$, as well as a \Kahler form $\omega_u$.
Note that the Riemannian metric $g = \omega_u(-,J_u-)$ is independent of $u$ and is Ricci-flat.

\paragraph{\autoref{thm:main_homogeneous} implies \autoref{thm:main_elliptic_counting}}
By \autoref{prop:elliptic_fibration_correspondence}, when the twistor plane $P$ is generic, elliptic fibrations at the complex structure $u\in \bS^2(P)$ are in correspondence with primitive integral isotropic vectors $e\in \Lambda_\bZ$ such that the orthogonal projection of $e$ to $P$ is proportional to $u$ with a positive factor.
By definition of the seminorm $\norm{-}_P$ the factor of proportionality between the projection of $e$ to $P$ and $u$ is exactly $\norm{e}_P$; it gives the volume of the elliptic fibration by \autoref{sssec:coh_volume_computation}.
Note that $-e$ will give an elliptic fibration at the complex structure $-u\in \bS^2(P)$).

We have thus established that in \autoref{thm:main_elliptic_counting} the counting function $N(V)$ is the same as the number of primitive integral isotropic vectors $e\in \Lambda_\bZ$ with $\norm{e}_{P}\leq V$.
The points $u_e \in \bS^2(P)$ where the elliptic fibrations occur are $\frac 1 {\norm{e}_{P}} e_P$ where $e_P$ is the projection of $e$ to $P$.
Therefore for generic planes $P$, \autoref{thm:main_homogeneous} and \autoref{thm:main_elliptic_counting} are, in fact, equivalent.

\paragraph{\autoref{thm:main_elliptic_counting} implies \autoref{thm:main_slag_counting}}
First, note that if a real torus $\bT^2\subset X$ is a complex genus one curve for one complex structure $u \in \bS^2(P)$, then it is also one at the opposite point $-u\in \bS^2(P)$.
By the correspondence between special Lagrangian and elliptic fibrations in \autoref{prop:hol_slag_corresp}, the same torus is a special Lagrangian (of appropriate phase) for any point on the equator corresponding to the two poles $u,-u$.
The volumes in both cases are equal to the Riemannian volume of the torus.
The elliptic fibration counting theorem then immediately implies the same statement about special Lagrangian fibrations, from the relation $N_{\textrm{slag}}(V) = \tfrac 1 2 N_{\textrm{elliptic}}(V)$ just established.
Note therefore that the leading term in \autoref{thm:main_slag_counting} is one half of that in \autoref{thm:main_elliptic_counting}.

The quantitative equidistribution statements about the poles/equ\-ators are seen to be equivalent via the Funk transform, discussed in \autoref{sssec:Funk_transform}.

\paragraph{Proof of \autoref{thm:main_riemannian}} % (fold)
\label{par:proof_of_thm:main_riemannian}
We will show that the Riemannian and holomorphic (resp. sLag) results count, in essence, the same objects.
In one direction, it is clear that any smooth fiber of an elliptic or sLag fibration gives an absolutely volume-minimizing surface in $X$, and all fibers are homologous.

Conversely, suppose that $T\to X$ is an absolutely volume-minimizing surface with $[T]^2=0$.
By the proof of \autoref{thm:main_elliptic_counting} above, there exists a complex structure on $X$ and an elliptic fibration of $X$ with class of fiber equal to $[T]$.
Let $\omega$ be the corresponding \Kahler form.
Then the volume of any fiber is equal to $[T]\cdot [\omega]$.

Now by Wirtinger's inequality we have
\[
  [T]\cdot [\omega]=\int_T \omega \leq \Vol(T).
\]
But $T$ is absolutely volume-minimizing and the fibers of the fibration already realize equality, so it must be that $\int_T \omega = \Vol(T)$.
From the equality case of Wirtinger's inequality, if follows that $T$ itself is a complex submanifold of $X$.

Because on a K3 surface a line bundle is determined by its class in $H^2$, it follows that the line bundles $\cO_X(T)$ and $\cO_X(E)$ are isomorphic, for any fiber $E$ of the elliptic fibration.
This implies that $T$ is the vanishing locus of a global section of $\cO_X(E)$, i.e. $T$ itself is a fiber of the elliptic fibration (see also the proof of \autoref{prop:elliptic_fibration_correspondence}). \qed

% paragraph proof_of_thm:main_riemannian (end)

\subsection{Analogies}

The counting result of \autoref{thm:main_slag_counting} can be seen as an analogue of counting closed billiard trajectories in the real $ 2 $-dimensional case.
Below are further analogies between the geometry and dynamics of surfaces in real and complex dimension $ 2 $.
Some of them have been suggested before, for instance in the work of Cantat \cite{Cantat} and McMullen \cite{McMullenSiegel}.
Others are suggested by the results of this paper.

Specifically, the action of complex automorphisms on the cohomology of a K3 surface and its relation to entropy goes back to Gromov \cite{Gromov} and has been studied further by Cantat \cite{Cantat}.
In particular, Cantat established the stable and unstable currents of a hyperbolic K3 automorphism as analogous to stable and unstable foliations for pseudo-Anosov mappings.
\Teichmuller space and the mapping class group could be interpreted on the K3 side as a period domain with an action of an arithmetic group.
Note that by the Torelli theorems (see \cite{K3sem}) there is a variety of period domains to consider on the K3 side; for example period domains of Hodge structures parametrize the complex structures on a K3, while period domains of twistor planes parametrize Ricci-flat metrics.

On the Riemannian part of the analogy, the flat metric on a Riemann surface is replaced by the Ricci-flat metric on a K3.
The holomorphic $ 1 $-form, when viewed as two real $1$-forms given by the real and imaginary parts, is replaced by the holomorphic $ 2 $-form (with real and imaginary parts) plus the \Kahler metric.
Straight lines are generalized to special Lagrangians.

The notion of ``angle'' on $ \bS^1 $ becomes now a point on the twistor sphere $ \bS^2 $.
As a flat surface is rotated, the metric and complex structure stay the same.
On the other hand, in a twistor family the Riemannian metric stays the same while the complex structure is changing.
This is because the complex structure is determined by only two of the three ``framing $2$-forms'' $\ip{\Re \Omega, \Im \Omega, \omega}$ (see also \autoref{sssec:diff_forms_twistor}).

Note that the variations of Hodge structures arising from K3 surfaces also admit an analogue of the Eskin--Kontsevich--Zorich \cite{EKZ_sum} formula for the sum of Lyapunov exponents.
This is established in \cite{sfilip_K3}.

A further application of Ricci-flat metrics to the dynamics of holomorphic automorphisms of K3 surfaces appears in \cite{FilipTosatti_KummerRigidity}.
In the reverse direction, dynamical techniques provide counterexamples to conjectured regularity of solutions to Monge--Amp\`ere equations in the boundary of the \Kahler cone in \cite{FilipTosatti_SmoothRough}.

\newcommand{\mc}[3]{\multicolumn{#1}{#2}{#3}}

\renewcommand{\arraystretch}{1.1}

\vskip 2em

%\begin{longtable}{p{0.51\textwidth} p{0.48\textwidth}}
	\begin{tabular}{p{0.51\textwidth} p{0.48\textwidth}}
		\toprule
		\multicolumn{1}{c}{\textbf{Riemann surfaces} } &  \multicolumn{1}{c}{ \textbf{K3 surfaces} }\\
		\midrule
		Mapping classes of diffeomorphisms: \newline pseudo-Anosov, reducible, periodic
		&
		Holomorphic automorphisms:\newline
		hyperbolic, parabolic, elliptic \\
		\midrule
		Entropy, action on curves & Entropy, action on ${H}^2$\\
		\midrule
		Stable and unstable foliations		 & Stable and unstable currents\\
		\midrule
    \Teichmuller space     & Period Domain(s)\\
    \midrule
		Flat metrics 								   & Ricci-flat (hyperk\"{a}hler) metrics\\
		\midrule
		Holomorphic $1$-form				& Holomorphic $2$-form\\
		\midrule
		Straight lines for the flat metric	  & Special Lagrangians\\
		\midrule
		Periodic trajectories 					   & Special Lagrangian tori\\
		\midrule
		Completely Periodic Foliations & Torus Fibrations\\
		\midrule
		$\bS^1$: directions for straight lines & $\bS^2$: twistor (hyperk\"{a}hler) rotation\\
		\midrule
		\multicolumn{2}{c}{Lyapunov exponents for families }\\
		\bottomrule
	 \end{tabular}
% \end{longtable}

\subsection{Remarks and References}

General references for the theory of K3 surfaces are the collection of notes \cite{K3sem} and the book by Huybrechts \cite{Huybrechts_K3}.
For Ricci-flat metrics and special Lagrangians, a general reference is the collection of notes \cite{Joyce_Gross_Huybrechts}.
The results necessary for this paper are recalled in \autoref{sec:K3_background}.

The counting techniques used in this paper go back to Eskin--Mc\-Mullen \cite{EskinMcMullen} and have been sharpened and quantified by Benoist--Oh \cite{Benoist_Oh} (see also the initial work of Duke--Rudnick--Sarnak \cite{DRS}).
The main results would extend to the case of general hyperk\"ahler manifolds, assuming one can establish an equivalence between classes of the fiber in a special Lagrangian fibration and null vectors in the closure of the positive cone.

In a different but related context, Tayou \cite{Tayou_HodgeLoci} has established a counting and equidistribution result for the Hodge locus in a polarized family of K3 surfaces.

\paragraph{Further work}
It would be interesting to extend the counting result from \autoref{thm:main_slag_counting} to all twistor families.

A Lagrangian fibration of a symplectic manifold induces on the base an integral-affine structure; the special Lagrangian fibrations counted in \autoref{thm:main_slag_counting} have as base the Riemann sphere $\bP^1(\bC)\cong \bS^2$.
The moduli space of integral-affine structures on the sphere carries an interesting geometry and similarities with strata of flat surfaces.
It is also connected to non-archimedean analytic geometry (see e.g. the work of Kontsevich--Soibelman \cite{Kontsevich_Soibelman}).
We hope to explore some of these analogies in future work.

\paragraph{Paper Outline}
\autoref{sec:K3_background} contains the necessary background on K3 surfaces, specifically the results needed to derive the counting results on K3 surfaces from the corresponding homogeneous results.
The section has an expository character and is meant for readers unfamiliar with K3 surfaces.

\autoref{sec:counting} contains the proof of \autoref{thm:main_homogeneous}, with a precise statement in \autoref{thm:homogeneous_counting}.
It is deduced from an equidistribution result in homogeneous spaces which is established in \autoref{sec:equidistribution}.

% Finally \autoref{app:explicit} contains explicit matrix descriptions of the Lie groups that appear in the arguments.

\paragraph{Acknowledgments}
I am grateful to Nicolas Bergeron and Carlos Matheus for their interest in this work, and for explicitly computing the constant $ \delta $ that appears in \autoref{thm:main_homogeneous}.

For discussions on the topic of this paper I am grateful to Yves Benoist, who in particular suggested a formulation along the lines of \autoref{thm:main_riemannian}, as well as Alex Eskin, Daniel Huybrechts, Martin M\"{o}ller, Misha Verbitsky, and Anton Zorich.
I received useful feedback and suggestions on a preliminary version of the text from Curt McMullen.

I am also grateful to the anonymous referee, whose remarks and suggestions significantly streamlined and improved the presentation in the paper.

This research was partially conducted during the period the author served as a Clay Research Fellow.

\section{Background on K3 surfaces}
\label{sec:K3_background}

This section collects classical results on K3 surfaces which were used in the derivation of \autoref{thm:main_elliptic_counting} and \autoref{thm:main_slag_counting} from the homogeneous counting result of \autoref{thm:main_homogeneous}.

\subsection{Basic definitions}

A detailed introduction to the theory of K3 surfaces is in the collected seminar notes \cite{K3sem}.
Huybrechts \cite{Huybrechts_K3} provides an updated account.

\begin{definition}
 A \emph{K3 surface} is a compact complex two-dimensional manifold $X$ that has trivial canonical bundle and is simply connected.
\end{definition}
Because the canonical bundle of a K3 is trivial, there is a unique up to scale holomorphic nowhere vanishing $2$-form $\Omega$.

% \begin{remark}\leavevmode
% \begin{enumerate}
%  \item That $X$ has trivial canonical bundle is equivalent to the existence of a nowhere vanishing holomorphic $2$-form, denoted $\Omega$ from now on.

%  \item The condition that $X$ is simply connected can be relaxed in various ways which lead to the same class of objects.
%  For example, vanishing of first Betti number suffices; for a purely algebraic definition, one can require that $H^1(X,\cO_X)=\{0\}$.
%  See \cite[Expos\'e III]{K3sem} or \cite[Ch. 1]{Huybrechts_K3} for more details.

%  \item It was proved by Siu \cite{SiuKahler} that every K3 surface admits a \Kahler form, thus Hodge theory applies in the usual form.
% \end{enumerate}
% \end{remark}

\subsubsection{Topological structure}
All K3 surfaces are diffeomorphic, and as a consequence the only non-trivial cohomology group is $H^2(X;\bZ)$.
Equipped with cup product, it is an even unimodular lattice of signature $(3,19)$ and this uniquely determines the isomorphism class of the lattice.
Denoting by $U$ the hyperbolic plane and by $(-E_8)$ the $E_8$ lattice with opposite sign of the quadratic form, the second cohomology group is non-canonically isomorphic to the K3 lattice
\begin{align*}
\bI_{3,19} := U^{\oplus 3}\bigoplus (-E_8)^{\oplus 2}.
\end{align*}
To simplify notation, the integral lattice is denoted by $ \Lambda_{\bZ} $ and its extension of scalars to a ring $R$ is denoted $\Lambda_R$.

% The corresponding groups of automorphisms preserving the quadratic form are $\Gamma := O(\Lambda_\bZ)$ and $G:=O(\Lambda_\bR)$.

Let $q$ denote the quadratic form on $\Lambda_\bZ$ (or its extension of scalars).
The induced bilinear form is
\begin{align*}
v\cdot w := \tfrac 12 \big(q(v+w) - q(v) - q(w) \big).
\end{align*}
Note that because $q$ is even, the bilinear form takes integer values.

\subsubsection{Hodge structure.}
The second cohomology group of a K3 surface $X$ carries a Hodge decomposition
\begin{align}
% \label{eqn:Hg_dec}
H^2(X;\bC) = H^{2,0}(X)\oplus H^{1,1}(X) \oplus H^{0,2}(X).
\end{align}
In this case, $H^{2,0} = \conj{H^{0,2}}$ and $H^{1,1}=\conj{H^{1,1}}$.
In particular, $H^{1,1}$ is the complexification of a real space, denoted $H^{1,1}_\bR(X)\subset H^2(X;\bR)$.

The space $H^{2,0}$ is $1$-dimensional and generated by the class of the holomorphic $ 2 $-form $\Omega$.
Expressing $ \Omega $ in holomorphic local coordinates gives:
\begin{align*}
[\Omega]\cdot [\Omega] = 0 \textrm{ and } [\Omega]\cdot [\conj{\Omega}] > 0.
\end{align*}
The complex conjugate of $\Omega$, denoted $\conj{\Omega}$, generates $H^{0,2}$.

\subsubsection{Riemann--Roch and Serre Duality}
Because on a K3 surface the canonical bundle is trivial, many cohomological statements take a particularly simple form.
The Rie\-mann--Roch theorem for a line bundle $\cL$ gives
\begin{align*}
 \chi(\cL) = \frac 12 [\cL]\cdot[\cL] + 2
\end{align*}
where $[\cL]$ denotes the first Chern class of the line bundle as an element of $\Lambda_\bZ$.
Serre duality implies $h^i(\cL)=h^{2-i}(\cL^\dual)$ where $\cL^\dual$ is the dual line bundle.
Using the definition of $\chi(\cL)$ and Serre duality, Riemann--Roch becomes:
\begin{align}
\label{eqn:RR_formula}
h^0(\cL) - h^1(\cL) + h^0(\cL^\dual) = \frac 12 [\cL]\cdot[\cL]+2.
\end{align}
Note that if $[\cL]^2\geq -2$ then at least one of $\cL$ or $\cL^\dual$ has a section, and in fact only one of them does unless $\cL$ is trivial.

\subsubsection{Adjunction}
Let now $C\subset X$ be a curve on a K3 surface.
The arithmetic genus of $C$ is defined as $p_a(C):= \rank H^1(C,\cO_C)$ and can be expressed by the \emph{adjunction formula}
\begin{align}
\label{eqn:adjunction_formula}
 2p_a(C)-2 = [C]\cdot[C]
\end{align}
As before, $[C]$ denotes the cohomology class of the curve $C$ on the surface.
When $C$ is smooth the arithmetic genus is the same as the topological genus.
When $C$ is not smooth the arithmetic genus equals the dimension of a space of meromorphic differentials on the normalization $\tilde{C}$ (whose genus can be smaller than the arithmetic genus of $C$).

\subsubsection{N\'eron--Severi group}
The Chern class of a holomorphic line bundle lies in a subgroup of the second cohomology, namely
$$
\NS(X) := H^{1,1}_\bR(X) \cap H^2(X;\bZ).
$$
Conversely, given any class $c\in \NS(X)$, by the Lefschetz $(1,1)$-theorem there will be a holomorphic line bundle $\cL_c$ with this Chern class.
The line bundle is unique, since a K3 surface is simply connected.
The cohomology class of any curve $C\subset X$ will also be in $\NS(X)$.

\subsubsection{The $(-2)$ curves}
\label{sssec:-2curves}
For any line bundle on a K3 surface, consider the common vanishing locus of all sections.
The $1$-dimensional part of this locus, if non-empty, is a union of irreducible curves $C$ with $[C]^2=-2$, possibly with multiplicities (see \cite[\S2.1.4]{Huybrechts_K3}).
An irreducible curve $C$ with $[C]^2=-2$ is called a $(-2)$ curve and is necessarily smooth and isomorphic to $\bP^1$.

\subsubsection{Elliptic fibrations}
Throughout this paper, an elliptic fibration will mean a holomorphic map $X\to B$ with general fiber a smooth connected genus $1$ curve (see \cite[Ch.~11]{Huybrechts_K3} for a more thorough discussion).
Two fibrations $X\to B$ and $X\to B'$ are equivalent if there is a map $B\to B'$ such that the corresponding diagram commutes.

\begin{proposition}
	\label{prop:elliptic_fibration_correspondence}
	Let $X$ be a K3 surface with no $(-2)$ curves.
	Then elliptic fibrations are in one-to-one correspondence with primitive integral elements $e\in \NS(X)$ with $e^2=0$, up to identifying $e$ with $ -e$.
	
	To an elliptic fibration one associates the class of a general fiber, which will be primitive integral isotropic, and conversely, for each such $e$ exactly one of $\pm e$ will occur as the fiber class of an elliptic fibration.
\end{proposition}

\begin{proof}
	If a K3 surface has an elliptic fibration, then the class of the general fiber is in the N\'eron--Severi group and is isotropic.
	The argument below will imply that the class is also primitive and $-e$ cannot occur as the class of a fibration.

	Let now $e\in \NS(X)$ be a primitive class with $e^2=0$; by \hyperref[eqn:RR_formula]{Riemann-Roch \eqref{eqn:RR_formula}} assume that the associated line bundle $\cL_e$ has at least two sections $s_1,s_2$ (otherwise pick $-e$).
	The two sections cannot have a common $1$-di\-men\-sional vanishing locus, since there are no $(-2)$ curves by assumption (see \autoref{sssec:-2curves}).
	The zero locus of each $s_i$ represents $e$ in (co)homology, and since $e^2=0$ the zero loci cannot intersect at all.
	This gives a well-defined holomorphic map $X\to \bP^1$.
	The fibers of the map are connected because $e$ is primitive (applying Stein factorization, components of disconnected fibers would have to be homologous).

	To conclude, by the \hyperref[eqn:adjunction_formula]{adjunction formula \eqref{eqn:adjunction_formula}} the fibers of the map have arithmetic genus $1$.
	Thus, the generic fiber is an irreducible smooth curve of genus $1$ and $X\to \bP^1$ is an elliptic fibration (see also \cite[Expos\'e IV, Sec. 3]{K3sem}).
\end{proof}

\subsection{Twistor families}
\label{ssec:twistor_family}

A detailed account of the material in this section is in the collected notes by Joyce, Gross, and Huybrechts \cite{Joyce_Gross_Huybrechts}.
For holonomy groups see Part I, and for hyperk\"ahler metrics and twistor families see Part III of loc. cit.

\subsubsection{Calabi--Yau metrics}
As a consequence of Yau's solution of the Calabi conjecture \cite{Yau_Ricci} any cohomology class in $\Lambda_\bR$ which is representable by \emph{some} \Kahler metric has a \emph{unique} representative with vanishing Ricci curvature.
Let therefore $(X,g,I)$ denote a triple where $X$ is a K3 surface with complex structure $I$ (viewed as an appropriate tensor) and $g$ is a Ricci-flat Riemannian metric with \Kahler form $\omega_I(-,-):=g(I-,-)$.

% \subsubsection{Holonomy}
% In terms of the Levi--Civita connection of a \Kahler metric, Ricci-flatness is equivalent to parallel transport preserving a holomorphic volume form on the canonical line bundle.
% For a K3 surface, this volume form is provided by the holomorphic $ 2 $-form $ \Omega $.

% In general, the group generated by parallel transport using the Levi--Civita connection of a Riemannian metric is called the \emph{holonomy} of the Riemannian metric.
% A metric on $X$ is \Kahler if its holonomy is in the unitary subgroup of the orthogonal group, in other words if parallel transport preserves a complex structure.
% Ricci flatness of the \Kahler metric on a complex surface is equivalent holonomy being contained in the subgroup $\SU(2)\subset \SO(4)$, i.e. parallel transport also preserves a complex volume form $\Omega$.

\subsubsection{Hyperk\"ahler structures}
\label{sssec:hyperkahler_def}
In Riemannian terms, the Ricci-flatness of the \Kahler metric on a K3 surface is equivalent to the holonomy group of the Riemannian metric being contained in $\SU(2)$.
Equivalently, the holomorphic $2$-form $\Omega$ is preserved by parallel transport using the Levi--Civita connection.

Because of the exceptional isomorphism $\SU(2) \cong \Sp(1)$ with the group of unit quaternions, parallel transport respects in addition to the complex structure $I$ another complex structure $J$.
The quaternionic commutation relation $K:=I\cdot J = -J\cdot I$ holds and this gives a family of complex structures $xI+yJ+zK$ with $x^2+y^2+z^2=1$, each preserved by parallel transport.

\noindent The result $(X,g,I,J,K)$ is the data of a \emph{hyperk\"ahler} manifold, where:
\begin{enumerate}
	\item $X$ is a smooth $4n$-manifold\footnote{For a K3 surface $ n=1 $.}, $g$ is a Riemannian metric on $X$.
	\item The tensors $I,J,$ and $K$ are integrable complex structures on $X$, satisfying the quaternionic commutation relations, and acting by isometries of the metric $g$.
	\item The 2-form $\omega_I(-,-):=g(I-,-)$ gives a \Kahler metric (i.e. $d\omega_I=0$), and similarly for $J$ and $K$.
\end{enumerate}

\subsubsection{Twistor families}
\label{sssec:twistor_families}
A hyperk\"ahler manifold as above gives rise to a natural family of complex manifolds over the unit sphere $\bS^2$.
The total space of the family is $\cX:=X\times \bS^2$ with the natural projection to $\bS^2$.
Given $t=(x,y,z)\in \bS^2$ a triple of real numbers satisfying $x^2+y^2+z^2=1$, define on the fiber $X_t$ the complex structure $I_t:=xI+yJ+zK$.
In the direction transverse to the fiber, the complex structure is the natural one induced from that on $\bS^2$.
A calculation shows that the natural projection map $\cX\to \bS^2=\bP^1(\bC)$ is holomorphic, and gives a holomorphic family of complex manifolds.

\subsubsection{The differential forms in a twistor family}
\label{sssec:diff_forms_twistor}
For a complex structure $I$ on $X$, denote the \Kahler form by $\omega_I(-,-):=g(I-,-)$.
Note that there is an $\bS^1$ worth of choices for the other complex structure $J$ satisfying $IJ=-JI$.
Indeed, any element in the equator perpendicular to $I\in \bS^2$ gives a complex structure with the required commutation relation.

The K3 surface $(X,I)$ also carries a nowhere vanishing holomorphic $2$-form $\Omega$, normalized by the requirement $\int_X \Omega\wedge \conj{\Omega}=2\int_X d\Vol(g)$.
This determines the form up to multiplication by a unit complex number.
The real and imaginary parts of $\Omega$ can now be expressed using the hyperk\"ahler structure, with $J$ and $K$ the other complex structures:
\begin{equation}
\begin{split}
\label{eqn:ReImOmega}
\Re \Omega (-,-) &= \omega_J(-,-) = g(J-,-)\\
\Im \Omega (-,-) &= \omega_K(-,-) = g(K-,-)
\end{split}
\end{equation}
In other words $\Omega = \omega_J + \sqrt{-1} \omega_K$.
If $\Omega$ is fixed, there is a unique pair of $J,K$ which satisfy \eqref{eqn:ReImOmega}.
The action of unit complex numbers by rotating $\Omega$ corresponds to the possible choices of pairs $(J,K)$ on the unit circle perpendicular to $I$ on $\bS^2$.

\subsubsection{Twistor planes}
A choice of \Kahler class $ [\omega] \in \Lambda_\bR $ gives a positive-definite $ 3 $-plane 
\[
 	P:= \operatorname{span} \ip{[\Re \Omega], [\Im \Omega], [\omega]}\subset \Lambda_\bR.
 \] 
Applying the twistor construction from \autoref{sssec:twistor_families} gives a family of K3 surfaces $ X_t $ with $ t \in \bP^1 \cong \bS^2(P) $, with \Kahler forms $ \omega_t $ and holomorphic $ 2 $-forms $ \Omega_t $ (with $ \Omega_t $ well-defined only up to a unit complex number).

Therefore a twistor family gives a positive-definite $ 3 $-plane $ P\subset \Lambda_\bR $ which will be called a twistor plane.
Conversely, given a positive-definite $ 3 $-plane inside $\bI_{3,19}\otimes \bR$ by the Torelli theorem there is an associated twistor family on some K3 surface, under some choice of isomorphism $\bI_{3,19}\otimes \bR \to H^2(X;\bR)$ (see e.g. \cite[Ch. 6]{Huybrechts_K3} or \cite[Ch. VII-X]{K3sem}).

Note that the base of the twistor family is naturally the unit $ 2 $-sphere in the $ 3 $-plane $ P $.
An element $ t\in \bS^2(P) $ gives a natural decomposition $ P = (\bR\cdot t) \oplus t^\perp $ and the \Kahler class $ [\omega_t] $ is tautologically specified by requiring it to be $ t $, while $ H^{2,0}\oplus H^{0,2} $ is just the complexification of $ t^\perp $.

\begin{definition}
	\label{def:generic_twistor}
	A twistor plane $P\subset \Lambda_\bR$ is \emph{generic} if for any $x\in \Lambda_\bZ$ with $x^2=-2$, there does not exist another integral vector $v\in \Lambda_\bZ$ such that $x$ and $v$ have proportional projection to $P$ along $P^\perp$.
	Equivalently, if a K3 surfaces in the twistor family contains a $(-2)$ curve, then the $(-2)$ curve generates the N\'eron--Severi group for that K3 surface.
\end{definition}
A particularly simple example of a generic twistor plane is one for which the N\'eron--Severi group has rank at most one for every member of the family.
The next result justifies the name ``generic'' for such twistor families by showing that they form a set of full measure.

\begin{proposition}
	\label{prop:generic}
	The positive-definite $ 3 $-planes $ P \subset \Lambda_\bR $ giving generic twistor families are the complement of countably many proper submanifolds in the Grassmannian of all positive-definite $ 3 $-planes.
\end{proposition}
\begin{proof}
	Let $ \Gr^+(3,\Lambda_\bR) $ be the Grassmannian of positive-definite $ 3 $-planes.
	For $ N\subset \Lambda_\bZ $ a rank $ 2 $ submodule which contains a vector with $x^2=-2$, consider the set $ D_N\subset \Gr^+(3,\Lambda_\bR) $ of $ 3 $-planes $ P $ for which $ N $ is contained in the N\'eron--Severi group for some element $ t$ in the twistor family $ \bS^2(P) $ (recall that we view $ t $ as a unit vector in $ P $).
	We will see that $ D_N $ is of codimension $ 1 $ in the Grassmannian, hence generic twistor families are the complement of the countable union of the $ D_N $ as $ N $ ranges over the countably many rank $ 2 $ submodules of $ \Lambda_\bZ $.

	A $ 3 $-plane $ P $ gives an orthogonal decomposition $ \Lambda_\bR = P \oplus P^\perp $.
	Denote orthogonal projection onto $ P $ by $ \pi_P $.
	Then $ N $ is contained in the N\'eron--Severi group of an element $ t\in P $ if and only if $ \dim \pi_P(N) \leq 1 $.
	Indeed, if $ \dim \pi_P(N) = 1 $ then $ N $ belongs to the N\'eron--Severi group associated to the image line, and if the dimension is zero (i.e. $ N\subset P^\perp $) then $ N $ is in the N\'eron--Severi group of each element of the corresponding twistor family.
	If the dimension of the projection is $ 2 $, then clearly $ N $ cannot occur in any N\'eron--Severi group.

	Since $ N $ is $ 2 $-dimensional, the condition $ \dim \pi_P(N)\leq 1 $ is a codimension $ 1 $ condition on $ P $, as $ P $ varies in the Grassmannian.
\end{proof}

% \begin{example}
% 	\leavevmode
% 	\begin{enumerate}
% 		\item Let $ X\subset \bP^3 $ be a quartic surface, e.g. $ x^4 - y^4 + z^4 - w^4 =0 $.
% 		Suppose that there exists a line $ l\subset X $ (e.g. $ x=y, z=w $).
% 		Consider the $ 1 $-dimensional family of planes $ H_t $ containing $ l $.
% 		Then $ X\cap H_t = l \cup C_t $ where $ C_t \subset H_t $ is a cubic curve (since $ X $ is quartic, and the line $ l $ accounts for one degree).
% 		This gives an elliptic fibration $ X \to \bP^1 $ (note that the line $ l $ maps by a degree $ 3 $ map to $ \bP^1 $, as each cubic curve will intersect $ l $ in $ 3 $ points).
% 		\item Consider now a bidegree $ (2,2,2) $ surface in $ \bP^1\times \bP^1\times \bP^1 $.
% 		The fibers of the projection to any one $ \bP^1 $-coordinate are $ (2,2) $ curves in $ \bP^1\times \bP^1 $, so have genus one.
% 		This gives three distinct elliptic fibrations, one for each choice of $ \bP^1 $.
% 		In fact, the surface has a large group of automorphisms which will produce more elliptic fibrations.
% 	\end{enumerate}
% \end{example}
\subsection{Special Lagrangians}
For a more detailed treatment of special Lagrangian manifolds and fibrations, see \cite[Parts 1 \& 2]{Joyce_Gross_Huybrechts}.

\begin{definition}[Special Lagrangian Manifold]
\label{def:sLag}
Let $X$ be a K3 surface, $\Omega$ its holomorphic $2$-form and $\omega$ a \Kahler form with induced metric $g$.
A real $2$-dimensional submanifold $L\subset X$ is a \emph{special Lagrangian submanifold} (abbreviated sLag) if the following conditions hold:
\begin{itemize}
 \item $L$ is Lagrangian for the symplectic form $\omega$, i.e. $\omega\vert_L=0$
 \item Restricting $\Omega$ to $L$ gives the Riemannian volume form of $L$ coming from $g$, i.e.
 \[
 \Omega\vert_L = d\Vol_g
 \]
 In particular, the imaginary part of $\Omega$ restricts to zero on $L$.
\end{itemize}
\end{definition}
For generic \Kahler metrics it can be difficult to find sLag submanifolds.
However, in the hyperk\"ahler case these exist in abundance, as will be discussed below.

\subsubsection{Special Lagrangian fibrations}
A \emph{sLag fibration} is a proper map $X \to B$ with $\dim_\bR B = \frac 12 \dim_\bR X$ such that all but finitely many fibers are smooth sLag manifolds; by the Arnold--Liouville theorem, the fibers are tori.
Typically the singular fibers are pinched tori and counted with multiplicities, there are $24$.
Indeed, a K3 surface has topological Euler characteristic $24$, a smooth torus has Euler characteristic $ 0 $, and a pinched torus has Euler characteristic $ 1 $.

Since a K3 surface is simply connected, any map to a topological surface of genus at least one must be homotopically trivial.
Therefore the base of a sLag fibration can only be a sphere.

The next result shows that in the hyperk\"ahler case, there is an equivalence between special Lagrangian fibrations and elliptic fibrations, in a different complex structure.

\begin{proposition}[Holomorphic $\leftrightarrow$ sLag correspondence]
	\label{prop:hol_slag_corresp}
 Let $(X,g,\allowbreak I,J,K)$ be a hyperk\"ahler $ 4 $-manifold (see \autoref{sssec:hyperkahler_def} for notation), with corresponding $2$-forms $\omega_\bullet$ ($\bullet\in \{I,J,K\}$).
 The following conditions are equivalent on a $ 2 $-dimensional submanifold $ L\subset X $:
 \begin{enumerate}
 	\item For the complex structure $ I $, holomorphic form $\Omega_I=\omega_J + \sqrt{-1}\omega_K$ and \Kahler form $\omega_I$, the submanifold $ L $ is sLag.
 	\item For the complex structure $ J $, the submanifold $ L $ is complex.
 	\item The two forms $ \omega_I, \omega_K $ vanish when restricted to $ L $, and $ \omega_J $ restricted to $ L $ coincides with Riemannian volume.
 \end{enumerate}
 In particular, special Lagrangian torus fibrations on $(X,\Omega_I,\omega_I)$ are in bijection with elliptic fibrations on $(X,J)$.
\end{proposition}%\leavevmode
\begin{wrapfigure}{r}{0.31\textwidth}
	\includegraphics[width=0.3\textwidth]{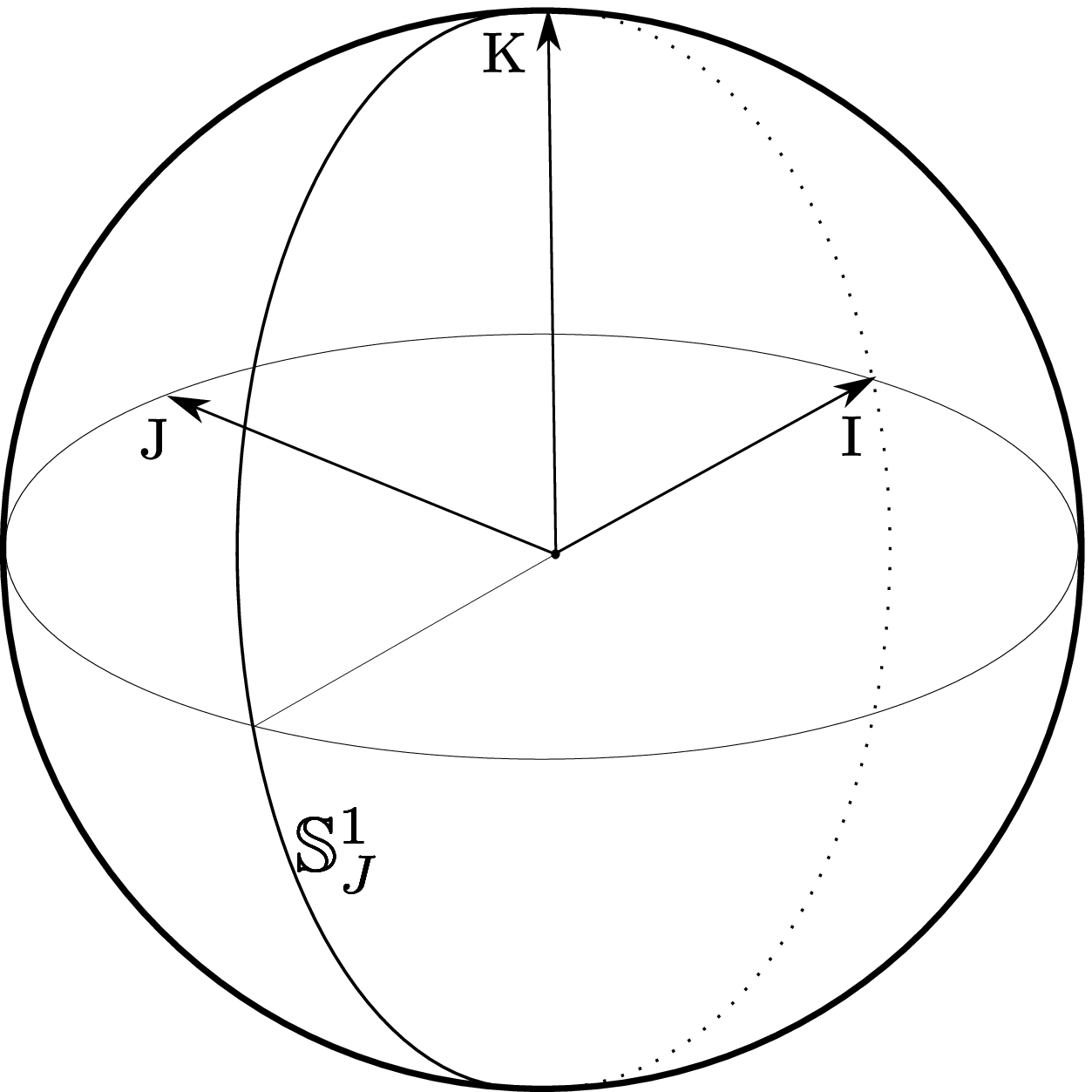}
\end{wrapfigure}\leavevmode
\begin{proof}
	The equivalence of (i) and (iii) follows from the definitions.

	Assume now (iii), i.e. $\omega_I\vert_L=\omega_K\vert_L = 0$ and $\omega_J\vert_L=d\Vol_g$.
	By Wirtinger's theorem, the last equality implies that the tangent spaces to $ L $ are at every point $ J $-complex subspaces of the ambient space.
	Therefore $ L $ is a $ J $-complex submanifold, so (ii) holds.

	Conversely, consider $X$ equipped with the complex structure $J$.
	Then the corresponding holomorphic $2$-form is $\Omega_J:=\omega_K + \sqrt{-1} \omega_I$ and the \Kahler form is $\omega_J$.
	Since $L\subset X$ is a complex submanifold for $J$ and $\Omega_J$ is a $ J $-holomorphic $2$-form, we have $\Omega_J\vert_L=0$; moreover $\omega_J\vert_L=d\Vol_g$ since $L$ is a complex submanifold, and $\omega_J$ is the \Kahler form.
	This gives (iii).
\end{proof}

% \subsubsection{Picture in the twistor sphere}
% \label{sssec:twistor_sphere}

% Recall that a hyperk\"ahler manifold gave rise to a twistor family over $\bS^2$ (see \autoref{ssec:twistor_family}).
% The complex structure $(X,J)$ corresponds to a point on the twistor sphere, and it has an orthogonal equator $\bS^1_J\subset \bS^2$.
% Every point on $\bS^1_J$ is of the form $I_t:=xI + z K$ with $x^2+z^2=1$.

% An elliptic fibration on $(X,J)$ gives a sLag fibration on each point $(X,I_t)$ on $\bS^1_J$.
% Indeed, once $I_t$ is picked on $\bS^1_J$, there is a unique $I_{t^\perp}$ such that $(I_t,J, I_{t^\perp})$ satisfy the quarternionic commutation relations.
% Then the holomorphic $2$-form on $(X,I_t)$ is given by $\Omega_t:=\omega_J + \sqrt{-1}\omega_{I_t^\perp}$ and $\omega_t = g(I_t-,-)$.

\begin{remark}
	\label{rmk:odd}
	Elliptic fibrations at the complex structure $ J $ are in natural bijective correspondence to ones at the complex structure $ -J $, with the same underlying set-theoretical fibrations but reversed orientation.
	The corresponding equator $ \bS^1_J $ is unique and parametrizes complex structure where the same fibration is sLag.
\end{remark}

\subsubsection{Volume of a fibration}
\label{sssec:volume_fibration}
For an elliptic fibration $X\xrightarrow{\pi} \bP^1$ on a K3 surface with \Kahler form $\omega$, define the volume of the fibration as
\begin{align}
 \label{eqn:vol_ell_fibr}
 \Vol(X\xrightarrow{\pi}\bP^1) := \int_{\pi^{-1}(pt)} \omega.
\end{align}
This is the volume of a fiber and is independent of the choice of fiber.

For a sLag fibration $X\xrightarrow{\pi}\bS^2$ of a K3 surface $X$ with holomorphic $2$-form $\Omega$, define the volume as
\begin{align}
 \label{eqn:vol_sLag_fibr}
 \Vol(X\xrightarrow{\pi}\bS^2) := \int_{\pi^{-1}(pt)} \Omega.
\end{align}
This is also equal to the Riemannian volume of any one fiber, since the fibers are special Lagrangian.

In both \autoref{eqn:vol_ell_fibr} and \autoref{eqn:vol_sLag_fibr} the volumes are computed by integrating a closed $ 2 $-form over the fiber of a fibration, and since the fibers are homologous the volume is independent of the choice of fiber.

\subsubsection{Cohomological computation of the volume}
\label{sssec:coh_volume_computation}
For a fibration, the class of a fiber $ [\pi^{-1}(pt)]\in H_2(X;\bZ) $ can be identified by cup product duality with a primitive isotropic $e\in \Lambda_\bZ$.
For an elliptic fibration, if $[\omega] \in \Lambda_\bR$ is the \Kahler class, the volume of the fibration is then
\begin{align}
\label{eqn:vol_coh}
\Vol(X \xrightarrow{\pi} \bP^1) = e\cdot  [\omega].
\end{align}
Similarly, for a sLag fibration with associated holomorphic $2$-form $\Omega$ we have
\begin{align}
	\label{eqn:vol_coh_slag}
	\Vol (X\xrightarrow{\pi} \bS^2) = e \cdot [\Omega].
\end{align}

In the case of a twistor family with twistor plane $P\subset \Lambda_\bR$, we have an orthogonal decomposition $\Lambda_\bR = P \oplus P^\perp$ and so $e = e_P \oplus e_{P^\perp}$.
If the elliptic and special Lagrangian fibrations correspond under \autoref{prop:hol_slag_corresp}, then both volumes are given by
\begin{align}
	\label{eqn:vol_universal}
	\Vol(X\xrightarrow{\pi}\bS^2=\bP^1) = e \cdot \frac{e_P}{\norm{e_P}} = \norm{e}_P
\end{align}
where $\norm{e}_{P}$ was defined in \autoref{ssec:results} by
\begin{align*}
  \norm{e}_P := \sup_{\substack{\kappa \in P ,\, \kappa^2=1}} \kappa \cdot e .
\end{align*}
Indeed, under the correspondence in \autoref{prop:hol_slag_corresp} the classes $[\omega]$ and $[\Omega]$ agree with $[\omega_t]$, where $t\in \bS^2(P)$ is given by $\frac{e_P}{\norm{e_P}}$.

\section{Counting}
\label{sec:counting}

This section contains the proof of \autoref{thm:main_homogeneous}, based on an equidistribution result established later in \autoref{sec:equidistribution}.
Although the statements are for a particular lattice $\Lambda_\bZ$ the methods extend with the same proof to a more general setting, see \autoref{sssec:generalizations}.
A standard reference for the notions of Lie theory used below is Bump \cite{Bump}.

After introducing the main objects and results, we connect the problem of counting isotropic vectors to counting intersections of a fixed subvariety with the orbit of another subvariety under the dynamics of a homogeneous flow.
The desired asymptotic count then follows from an equidistribution statement for the dynamics of the flow.

\subsection{Statements}

\subsubsection{Setup}
Keeping the notation as in the previous section, fix a lattice $\Lambda_\bZ$ isomorphic to the second cohomology of a K3 surface (see however \autoref{sssec:generalizations} for a more general statement).
Extending scalars to $\bR$ gives the vector space $\Lambda_\bR$ equipped with an inner product of signature $(3,19)$.

Fix a positive-definite $3$-plane $P\subset \Lambda_\bR$ with orthogonal complement $P^\perp$ giving a decomposition $\Lambda_\bR = P \oplus P^\perp$.
Using this decomposition, for any vector $v\in \Lambda_\bR$ denote by $v_{P}$ and $v_{P^\perp}$ the respective coordinates.
For a vector $v\in P$ denote its projection to $\bS(P)$, the unit sphere of $P$, by $\dot{v}$:
\begin{align}
\dot{v}:= \frac {v}{(v^2)^{1/2}} \textrm{ so that }(\dot{v})^2 = 1.
\end{align}
The same notation will apply to vectors in $P^\perp$ and to its unit sphere.

Denote the set of non-zero primitive integral isotropic vectors by $\Lambda_\bZ^0$ and pick an element $v$ in it.
Using the notation just introduced, we have $e = e_P + e_{P^\perp}$ and $ \dot{e_P} = \frac 1 {\norm{e}_{P}} e_P $, where the seminorm $\norm{-}_P$ was introduced in \autoref{ssec:results}.

The goal of this section is to prove the following theorem, which in the particular case of a constant weight function gives \autoref{thm:main_homogeneous}.
\begin{theorem}
	\label{thm:homogeneous_counting}
	Let $w:\bS^2(P)\to \bR$ be a smooth weight function on the unit sphere of $P$.
 	Then the following asymptotic formula holds as $ V\to \infty $:
	\begin{align}
	\label{eqn:homogeneous_asymptotic_poly}
		 \sum_{\substack{e \in \Lambda_\bZ^0\\ \norm{e}_{P} \leq V }} w(\dot{e}_P) = C \cdot \left( \int_{\bS^2}w \right) \cdot V ^{20} + O_w\left(V^{20 - \delta}\right).
	\end{align}
	The implied constant in the error term $O_w(-)$ depends on a Sobolev norm of the function $w$.
\end{theorem}

% \subsubsection{Setup}
% Let $ \Lambda_\bZ $ be a lattice isomorphic to $ H^{2}(X;\bZ) $ for a K3 surface $ X $.
% It will be assumed throughout this section that all K3 surfaces are marked (in the sense of \autoref{sssec:marking}) by $ \Lambda_\bZ $, so that we can speak of subspaces inside $ \Lambda_\bZ $, or its real or complex elements $ \Lambda_\bR, \Lambda_\bC $.

\subsubsection{Generalizations}
    \label{sssec:generalizations}
The proof of \autoref{thm:homogeneous_counting} applies to a more general setting, with the following modifications.
The lattice $\Lambda$ need not be even and can have any signature $(p,q)$ with $p,q>1$ avoiding the reducible case $p=q=2$.
The exponent $20$ in the leading term should be then changed to $p+q-2$.

Additionally, the summation would be not over all primitive isotropic vectors, but only over a fixed $\Gamma$-orbit of such.
By a theorem of Borel--Harish-Chandra \cite{Borel_Harish_Chandra} the collection of all primitive isotropic vectors is divided into finitely many $\Gamma$-orbits.

\subsubsection{Volumes and Tamagawa numbers}
\label{sssec:tamagawa}
The constant $C$ in the main term of \autoref{eqn:homogeneous_asymptotic_poly} is of the form
\[
	C = \frac{\Vol Y}{20\cdot \Vol X} \in \frac 1 {\pi^{20} \cdot \zeta(11)} \cdot \bQ
\]
and involves the volumes of two locally homogeneous spaces $X,Y$ defined in \autoref{sssec:loc_homg_spaces}.
Up to rational factors, the volumes can be evaluated explicitly as follows (see \cite{Kneser} for a discussion).

First, find a $\bQ$-form of the relevant orthogonal groups and apply the following type of result:
\begin{align}
\label{eqn:product_formula}
	\Vol(\cG(\bR)/\Gamma ) \cdot \prod_p \Vol \cG(\bZ_p) = \tau(\cG)
\end{align}
where $\cG$ is a $\bQ$-algebraic group and $\tau(G)$ is the Tamagawa number of $\cG$ (an integer, typically $1$ or $2$).
The volumes are normalized using the Tamagawa measure, involving an explicit rational differential form.

The volumes of $\cG(\bZ_p)$ equal $p^{-\dim G}\cdot \# \cG(\bF_p)$ for all sufficiently\footnote{in fact, all $p$ since our quadratic form is unimodular} large primes $p$.
The number of points of the orthogonal groups over finite fields can be found in \cite[\S~1.4]{Carter} from which it follows that for an orthogonal group of rank $2l$ the volumes can be expressed in terms of Riemann zeta values, up to rational factors:
\[
	\prod_p \Vol \cG(\bZ_p) \approx \left( \zeta(2)\cdots \zeta(2l-2) \cdot \zeta(l) \right)^{-1}
\]
The sought-after constant $C$ is the ratio of two such volumes, for orthogonal groups of rank $22$ and $20$ respectively, so up to rational factors
\begin{align}
	C \approx (\pi^{20} \cdot \zeta(11))^{-1}
\end{align}
where even zeta values are evaluated explicitly in terms of $\pi$ (and Bernoulli numbers).
Note that the group defining $Y$ is not orthogonal, but rather an extension of an orthogonal group by a unipotent, but the unipotent part will only contribute rational factors to the volume.
To determine the exact rational factor of $C$, one would have to follow through all the normalizations involved in the definition of the Tamagawa measure and the measures in this text, as well as the isogeny classes of orthogonal groups involved.

\subsubsection{Polynomial vs. Exponential formulation}
The proof of \autoref{thm:homogeneous_counting} will involve a flow on a homogeneous space, with time parameter $ t\in \bR $.
Because of the appearance of exponentials, to avoid ambiguity primitive null vectors will be denoted by $v$ instead of $e$ from now on.
Associated to each null vector $ v $ there will be a ``hitting time'' $ t(v) $ (see \autoref{sssec:hitting_time}) related to the seminorm by
\begin{align*}
\exp{(t(v))} = \norm{v}_P.
\end{align*}
In exponential form, the counting in \autoref{eqn:homogeneous_asymptotic_poly} can be rewritten (using the change of variables $V=e^T$) as
\begin{align}
\label{eqn:homogeneous_asymptotic_exp}
\sum_{\substack{v \in \Lambda_\bZ^0\\ t(v) \leq T }} w(\dot{v}_P) = C \cdot \left( \int_{\bS^2}w \right) \cdot e^{20\cdot T}+ O_w\left(e^{(20 - \delta)\cdot T}\right).
\end{align}

\subsubsection{Proof sketch}
We will consider a homogeneous space $ X = \leftquot{\Gamma}{G} $ and a homogeneous subspace $ Y \subset X $ (see \autoref{ssec:preliminaries} for more precise notation).
For a $ 1 $-parameter homogeneous flow $ a_t $ on $ X $, the translates of $ Y $ will quantitatively equidistribute in $ X $; this is established in \autoref{sec:equidistribution} below, but will be assumed in this section.

Inside $ X $ there will be a compact subset $ K $ (the image of the maximal compact in $ G $) and furthermore a compact subgroup $ K_v\subset K $ with $ \dim K + \dim Y - \dim K_v + 1 = \dim X $.
The times $ t $ when the $ a_t $-translates of $ Y $ intersect $ K $ will correspond to $ t $'s in \autoref{eqn:homogeneous_asymptotic_exp} which occur in the summation, and the intersection $ a_t Y \cap K $ will be along a $ K_v $-coset inside $ K $.
The function $ w $ that appears in the count can be thickened to a function on $ X $ supported near $ K $, and the quantitative equidistribution of the $ Y $-translates will give the desired count.

The equidistribution results proved below extend to functions $ w:\bS^2(P)\times \bS^{18}(P^\perp)\to \bR $ sampled at $\dot{v}_P$ and $\dot{v}_{P^\perp}$.
However, $ v_{P^\perp} $ does not seem to have a geometric interpretation in the setting of K3 surfaces so we do not pursue this direction.

\subsection{Some preliminaries}
\label{ssec:preliminaries}

\subsubsection{Groups}
\label{sssec:groups}
Denote the orthogonal groups by
\begin{equation}
\begin{split}
& G:=\Orthog(\Lambda_\bR) \textrm{ and }\Gamma:=\Orthog(\Lambda_\bZ)\\
& K:=\Orthog(P)\times \Orthog(P^\perp) \textrm{ maximal compact in }G.
\end{split}
\end{equation}
Fix a primitive integral isotropic vector $v\in \Lambda^0_\bZ$ with corresponding stabilizers
\begin{equation}
\begin{split}
& H_v:=\Stab_{G}(v) \textrm{ and }\Gamma_v:=\Stab_{\Gamma}(v)\\
& K_v:=\Stab_{K}(v) \textrm{ maximal compact in } H.
\end{split}
\end{equation}
Note that $H_v$ is a $\bQ$-subgroup of $G$ and $\Gamma_v$ is a lattice inside it.

\subsubsection{Locally homogeneous spaces}
\label{sssec:loc_homg_spaces}
With the notation as in \autoref{sssec:groups}, define the following quotient spaces:
\begin{equation*}
\begin{split}
X &:= \leftquot{\Gamma}{G}\\
X_v &:= \leftquot{\Gamma_v}{G}\\
Y &:= \leftquot{\Gamma_v}{H_v}
\end{split}
\qquad \qquad
\begin{tikzcd}
Y \arrow[r, hook] \arrow[rd, hook] & X_v \arrow[d] \\
& X
\end{tikzcd}
\end{equation*}
The quotients $Y$ and $X$ have finite volume, whereas $X_v$ is infinite-volume.
Note that $X$ and $X_v$ carry a right $G$-action.

\subsubsection{Sphere parametrization}
\label{sssec:sphere_param}
Using the compact groups $K\supset K_v$ yields a natural parametrization of the unit spheres:
\begin{align}
\begin{split}
\rightquot{K}{K_v} & \leftrightarrow \bS(P)\times \bS(P^\perp)\\
[k] & \mapsto [k]\cdot \left( \dot{v}_{P} , \dot{v}_{P^\perp} \right).
\end{split}
\end{align}

\subsubsection{Primitive null vector parametrization}
\label{sssec:intg_null_param}
Any primitive null vector $w\in \Lambda_\bZ^0$ is in the orbit of $v$ under $\Gamma$.
One way to see this is to invoke a theorem of Eichler \cite[\S 10.4]{Eichler} which implies that two primitive vectors of the same length in an even unimodular lattice $ \Lambda $ are related by an orthogonal transformation, provided the lattice contains two hyperbolic planes.
In the case at hand, a simple direct proof is possible; I am grateful to the anonymous referee and Curt McMullen for pointing out the following argument.

Let $w\in \Lambda^0_\bZ$ be a primitive null vector; there exists some $x\in \Lambda_\bZ$ such that $x\cdot w = 1$ since $w$ is primitive, and $x^2=2k$ with $k\in \bZ$ since $\Lambda$ is an even lattice.
Now the vectors $w, x_1:=x-kw$ span a hyperbolic plane $W\subset \Lambda_\bZ$, i.e. $w^2={x_1}^2=0$ and $w\cdot x_1 = 1$.
The orthogonal complement $W^\perp$ is an even unimodular lattice of signature $(2,18)$ and hence unique up to isomorphism.
Applying the same argument to another null vector $w_1$ produces another decomposition
\[
 	W_1 \oplus W_1^\perp = \Lambda = W \oplus W^\perp.
 \] 
By the uniqueness of even unimodular lattices of indefinite signature, there is an isomorphism $W_1^\perp \to W^\perp$ and there is also an isomorphism of hyperbolic planes $W_1 \to W$ taking $w$ to $w_1$.
This gives an automorphism of $\Lambda$ which takes $w$ to $w_1$.

In conclusion, the $\Gamma$-orbit of $v$ gives a natural parametrization of the primitive integral vectors:
\begin{align*}
\begin{split}
\rightquot{\Gamma}{\Gamma_v} & \leftrightarrow \Lambda_\bZ^0\\
[\gamma] &\mapsto [\gamma]\cdot v.
\end{split}
\end{align*}

\subsubsection{Cosets}
Each null vector $f\in \Lambda_\bR$ gives a point on each of the unit spheres $(\dot{f}_{P}, \dot{f}_{P^\perp})\in \bS(P)\times \bS(P^\perp)$.
Using \autoref{sssec:sphere_param}, this is the same as a right $K_v$-coset $C_f\subset K$.
Since by \autoref{sssec:intg_null_param} each integral primitive null vector is of the form $[\gamma]v$ for $[\gamma]\in \rightquot{\Gamma}{\Gamma_v}$, denote by $C_{[\gamma]}\subset K$ the corresponding right $K_v$-coset.

\subsection{Some dynamics}

\subsubsection{The key one-parameter subgroup}
Associated to $v$ and $P$ is another isotropic vector $v'$ defined by
\begin{align*}
\begin{split}
v & = v_{P} \oplus v_{P^\perp}\\
v' & = v_{P} \oplus -v_{P^\perp}.
\end{split}
\end{align*}
Define the one-parameter subgroup $a_t\subset G$ by
\begin{align}
\label{eqn:a_t_def}
\begin{split}
a_t \cdot v  &= \exp({-t}) \cdot v\\
a_t \cdot v' &= \exp(t)\cdot v'
\end{split}
\end{align}
and acting as the identity on $ (v\oplus v')^\perp $.

\subsubsection{Hitting time}
\label{sssec:hitting_time}
The vectors $C_{[\gamma]} v$ and $[\gamma]v$ have proportional coordinates in $P$ and $P^\perp$.
Define $t([\gamma])\in \bR$ by the requirement
\begin{align}
\nonumber e^{2\cdot t([\gamma])} (v)_{P}^2 & = \left( [\gamma]v \right)_P^2
\intertext{so that we have the identity}
\label{eqn:C_gamma_a}
C_{[\gamma]}a_{-t([\gamma])} v &= [\gamma] v.
\end{align}
Indeed, both the left and right-hand side of \autoref{eqn:C_gamma_a} have proportional coordinates in $P$ and $P^\perp$ and the action of $a_\bullet$ ensures that the lengths of the coordinates also agree.

For the next proposition recall that $H_v$ is the stabilizer of $v$ in $G$.
\begin{proposition}
	\label{prop:C_gamma_inters}
	We have the equality of subsets of $G$:
	\begin{align*}
	C_{[\gamma]} = K \cap \left( [\gamma] H_v a_{t([\gamma])} \right).
	\end{align*}
\end{proposition}
\begin{proof}
	Rewrite the identity \eqref{eqn:C_gamma_a} as
	\begin{align*}
	[\gamma]^{-1} C_{[\gamma]} \cdot a_{-t([\gamma])} v =  v
	\end{align*}
	so that
	\begin{align*}
	[\gamma]^{-1} C_{[\gamma]} \cdot a_{-t([\gamma])} \subset H_v
	\textrm{ or equivalently }
	C_{[\gamma]} \subset [\gamma] H_v a_{t([\gamma])}.
	\end{align*}
	By construction $C_{[\gamma]}\subset K$ so one direction of the inclusion follows.

	For the reverse inclusion, note that any element of $[\gamma] H_v a_{t([\gamma])}$ takes $v$ to a vector proportional to $[\gamma]v$, but preserving the lengths of the $P$ and $P^\perp$ coordinates.
	Requiring that it is additionally in $K$ implies that it is contained in $C_{[\gamma]}$.
\end{proof}

\subsubsection{Summary}
\label{sssec:gamma_datum}
Every $[\gamma]\in \rightquot{\Gamma}{\Gamma_v}$ has the following data attached:
\begin{enumerate}
	\item A primitive integral null vector $ \gamma \cdot v\in   \Lambda^0_\bZ$.
	\item A right $K_v$-coset $C_{[\gamma]}\subset K$ whose action on $v$ aligns its $P\oplus P^\perp$ coordinates with those of $[\gamma]v$.
	\item A scalar $t([\gamma])\in \bR$ equal to the log of the ratio of lengths of the $P$ coordinates of $v$ and $[\gamma]v$.
	\item By \autoref{prop:C_gamma_inters} we have the equality of sets
	\begin{align}
	\label{eqn:C_gamma_inters}
	C_{[\gamma]} = K \cap \left( [\gamma] H_v a_{t([\gamma])} \right)
	\end{align}
	% \item For distinct cosets $[\gamma]$ and $[\gamma']$ the values of $t([-])$ could agree.
	% However, the cosets $C_{[-]}$ are necessarily distinct since $[\gamma]$ and $[\gamma']$ correspond to primitive vectors, and there is a unique one on each line.
\end{enumerate}
Any one of the following uniquely determines the other: the coset $[\gamma]$, the primitive null vector $[\gamma]v$, the right $K_v$-coset $C_{[\gamma]}$.
The only non-trivial claim to check is that if two primitive null vectors $v_1,v_2$ have positively proportional projections to $P$ and $P^\perp$ (this determines the cosets $C_{[\gamma_i]} $) then they are in fact equal.
But since $v_{i,P}^2 = - v_{i,P^\perp}^2 $ it follows that their proportionality factors are the same, so $v_1 = \lambda v_2$, which contradicts the primitivity of the $v_i$, unless they are equal.

\subsection{Natural measures and Intersections}
\label{ssec:natl_msrs}

For the equidistribution and counting results used below, we will need to fix appropriate measures on the homogeneous spaces.
This section describes the normalizations.
Additionally, we describe the intersections of the translates $ Y \cdot a_t \subset X$ with the compact group $ K\subset H $.

\subsubsection{Constant-volume measure}
\label{sssec:cont_vol_msr}
The coset $Y a_t \subset X$ carries two natural measures.
The first one, denoted $\mu_{Y a_t}$, comes from pushing forward the Haar measure on $Y$ via $a_t$:
\begin{align*}
\begin{split}
(-)\cdot a_t\colon Y &\to Y a_t\\
\mu_Y &\mapsto \mu_{Y a_t}
\end{split}
\end{align*}
In particular, the total volume of $\mu_{Y a_t}$ is independent of $t$.

\subsubsection{Homogeneous measure}
\label{sssec:homg_msr}
The second measure, denoted $\mu_{Y_t}$, is defined as the induced measure on $ Y a_t$ viewed as a homogeneous space under $H_v$.
Indeed, the right action of $H_v$ on $Y a_t$ gives the identification
\begin{align*}
Y a_t \cong \leftquot{a_{-t}\Gamma_v a_t }{H_v}.
\end{align*}

\begin{proposition}
	\label{prop:msr_scaling}
	The two measures are related by a scaling factor:
	\begin{align*}
	\mu_{Y a_t} \cdot e^{20\cdot t} = \mu_{Y_t}
	\end{align*}
\end{proposition}
\begin{proof}
	Consider the conjugation actions
	\begin{align}
	\begin{split}
	\Ad_{a_t}:H_v & \to H_v\\
	h & \mapsto a_{-t} h a_t
	\end{split}
	%\intertext{and its Lie algebra version}
	\begin{split}
	& \ad_{a_t}:\frakh_v \to \frakh_v\\
	& \textrm{on the Lie algebra.}
	\end{split}
	\end{align}
	Looking at the tangent space to the identity in $H_v$, it follows that
	\begin{align}
	\mu_{Y a_t} \cdot \det\left(\ad_{a_t}\right) = \mu_{Y_t}
	\end{align}
	But the adjoint action of $ a_t $ on $ H_v $ preserves the semisimple part and expands the unipotent part by a factor of $e^t$ (see also \autoref{ssec:Lie_alg_dec}).
	The dimension of the unipotent part is $ 20 $, so the determinant is $ e^{20\cdot t} $.
\end{proof}

\subsubsection{Intersections}
Consider the set $\Gamma\cdot K\subset G$ and its projection to $X_v$, denoted $\leftquot{\Gamma_v}{\Gamma}\cdot 	K\subset X_v$.
The following is a reformulation of \autoref{eqn:C_gamma_inters}.
\begin{proposition}
	\label{prop:inters_biject}
 The cosets $[\gamma]\in \rightquot{\Gamma}{\Gamma_v}$ are in bijection with intersections between $(\leftquot{\Gamma_v}{\Gamma})\cdot K$ and translates $Y\cdot a_t$ inside $X_v$.
 The intersection occurs at $t=t([\gamma])$ and consists of the set $[\gamma]^{-1}C_{[\gamma]}$.

 Equivalently, projecting $X_v$ down to $X$, the bijection is between cosets $[\gamma]$ and intersections of $Y \cdot a_t$ with $K\subset X = \leftquot{\Gamma}{G}$.
\end{proposition}
\begin{proof}
 The equivalence of the two statements is immediate.
 To see the first statement, apply $[\gamma]^{-1}$ to \autoref{eqn:C_gamma_inters}.
\end{proof}

\begin{remark}
\label{rmk:genericity_K}
In \autoref{prop:inters_biject} the intersections have to be understood with multiplicity.
For a fixed time $t$ there could be several $[\gamma]$ for which $t=t([\gamma])$, but they all give rise to different cosets $C_{[\gamma]}$.
\end{remark}

\subsection{Decompositions of the Lie algebra}
\label{ssec:Lie_alg_dec}
This section contains the necessary Lie algebra decompositions for subsequent arguments.
% They are made explicit, in terms of matrices, in \autoref{app:explicit}.
The maximal compact $K\subset G$ gives a Cartan involution $\theta$ on $G$ and on $\frakg:=\Lie G$.
A subalgebra of $\frakg$ is \emph{symmetric} if it is $\theta$-invariant.

The one-parameter subgroup $a_t$ defined in \eqref{eqn:a_t_def} is symmetric.
Let $\fraka$ denote its Lie algebra, which is $1$-dimensional, and let $\fraka^+\subset \fraka$ be the fixed positive Weyl chamber, corresponding to the positive values of the parameter $t$ in \eqref{eqn:a_t_def}.

\subsubsection{Weight decomposition}
\label{sssec:weight_dec}
The adjoint action of $\fraka$ on $\frakg$ has the weights $\{-1,0, +1\}$ and gives the decomposition
\begin{align*}
 \frakg = \frakn^- \oplus \left( \frakm \oplus \fraka \right) \oplus \frakn^+.
\end{align*}
The subalgebra $ \frakm $ is picked so that it is orthogonal (for the Killing form) to $ \fraka $.
The involution $\theta$ exchanges $\frakn^+$ and $\frakn^-$.

\subsubsection{Reductive part}
Note that $\frakm$ itself decomposes:
\begin{align*}
 \frakm = \frakk_v \oplus \frakp_\frakm \textrm{ where }\frakk_v = \frakm \cap \frakk = \Lie K_v \textrm{ and }\theta\vert_{\frakp_\frakm} = -\id.
\end{align*}

\subsubsection{Horospherical part}
\label{sssec:He_dec}
Recall that $H_v:=\Stab_{v}G$.
Then its Lie algebra is:
\begin{align}
\label{eqn:lie_h0_dec}
 \Lie H_v = \frakh_v = \frakm \oplus \frakn^+ = \left(\frakk_v \oplus \frakp_\frakm\right) \oplus \frakn^+.
\end{align}

\subsubsection{A transverse to the compact part}
\label{sssec:lie_transv_cpct}
The following decomposition will be useful when constructing a thickening of the compact subgroup $K\subset G$:
\begin{align}
\label{eqn:lie_transv_cpct}
 \frakg = \frakk \oplus \left(\frakp_\frakm \oplus \frakn^+ \right) \oplus \fraka.
\end{align}
Indeed, this follows from the more refined composition
\begin{align}
\label{eqn:lie_transv_cpct_refined}
 \frakg & = \frakk^c \oplus \left(\overbrace{\underbrace{\frakk_v \oplus \frakp_\frakm}_{\frakm} \oplus \frakn^+}^{\frakh_v} \right) \oplus \fraka\\
\intertext{where $\frakk = \frakk^c \oplus \frakk_v$, $ \Lie K_v = \frakk_v $ and}
 \nonumber \frakk^c & = \{ x\oplus -\theta(x) \in \frakn^-\oplus \frakn^+ \,\vert\, \forall x\in \frakn^- \}.
\end{align}
This shows that locally on $ G $, the right $ a_t $-translates of $ H_v $ intersect $ K $ along right $ K_v $-cosets.

\subsubsection{Norms}
The Cartan-Killing pairing on $\frakg$ is indefinite.
However, the choice of a maximal compact $K$ (equivalently, of Cartan involution $\theta$) determines a natural metric on $\frakg$.
In particular, this determines natural volume forms on all subalgebras.

\subsubsection{Thickening of $ K $}

The next construction provides a good thickening of functions on the compact $K\subset G$, with control on the Sobolev norm.
Sobolev norms on the non-compact space $ X $ are defined in \autoref{ssec:spectral_th}, but since $ K $ is compact and the thickening will be supported in the $ \epsilon $-neighborhood of $ K $, the usual definitions of Sobolev norms apply.

\begin{proposition}
\label{prop:w_epsilon}
 Let $w\colon K\to \bR$ be a smooth function.
 For any sufficiently small $\epsilon>0$ there exists a function $w_\epsilon \colon G\to \bR$ with the following properties:
 \begin{enumerate}
  \item The support of $w_\epsilon$ is in the $\epsilon$-neighborhood of $K$ and the masses of $ w $ and $ w_\epsilon $ agree:
  \begin{align}
   \int_{G} w_\epsilon \, d\mu_G = \int_K w\, d\mu_K.
  \end{align}
  \item The $ l $-th Sobolev norm of  $w_\epsilon$ satisfies the following bound, for some constant $ d_l $ depending on $ l $ and implicit constant depending on $ w $:
  \begin{align}
  \label{eqn:sobolev_bd_cpct_thick}
   \norm{w_\epsilon}_l \ll_{w} \left(\epsilon\right)^{-d_l}.
  \end{align}
  \item Fix a right  $ K_v $-coset $kK_v\subset K$ and consider the family of orbits $kH_v \cdot a_t\subset G$ with $t\in(-\epsilon,\epsilon)$ (recall that $K_v\subset H_v$, see also \eqref{eqn:lie_transv_cpct_refined}).
  We then have the equality of integrals:
  \begin{align}
  \label{eqn:w_eps_coset_average}
   \int\displaylimits_{-\epsilon}^\epsilon\! dt \! \left(\, \int\displaylimits_{k H_v \cdot a_t}\! w_\epsilon \, d\mu_{H_t}\right) = \int\displaylimits_{kK_v}\! w \, d\mu_{K_v}
  \end{align}
  Here $d\mu_{H_t}$ denotes the measure on $H_v\cdot a_t$ when viewed as a homogeneous $H_v$-orbit (see \autoref{sssec:homg_msr} and \autoref{prop:msr_scaling}).
 \end{enumerate}
\end{proposition}

\begin{proof}
 \autoref{eqn:lie_transv_cpct} provides a transverse chart to $K$, i.e. a diffeormorphism onto a local neighborhood of $K$:
 \begin{align}
 \label{eqn:k_thick_coords}
 \begin{split}
  K\times \frakp_{\frakm,\delta} \times \frakn^+_{\delta} \times \fraka_{\delta} & \to G\\
  (k,p,n,a) &\mapsto k\cdot \exp(p)\cdot \exp(n)\cdot \exp(a)
 \end{split}
 \end{align}
 The subscripts $(\bullet)_\delta$ denote $\delta$-neighborhoods of the corresponding vector spaces, for some uniform $ \delta>0 $.
 Note that Haar measure on $ G $ pulls back to Haar measure on $ K $ times some measure on $ \frakp_{\frakm,\delta} \times \frakn^+_{\delta} \times \fraka_{\delta} $ which is essentially Euclidean, up to uniform constants.

 Fix now a bump function $ \chi $ on $ \frakp_{\frakm} \times \frakn^+ \times \fraka $ and let $ \chi_\epsilon $ denote its rescaling:
 \begin{align}
 \label{eqn:chi_epsilon}
 \chi_\epsilon (k,x) := f(k,\epsilon) \epsilon^{-c_1(l)} \cdot \chi\left( \frac x \epsilon\right)\quad \textrm{for }k\in K
 \end{align}
 where the prefactor $ f(k,\epsilon)\epsilon^{-c_1(l)} $ is chosen such that $ \int \chi_\epsilon (k,x) dx=1 $ for all $ k\in K $.
 Note that because the pull-back of Haar measure is Euclidean up to bounded factor, the multiplier $ f(k,\epsilon) $ is also uniformly bounded.

 Next, the $ l$-th Sobolev norm of $ \chi_\epsilon $ is bounded by $ \epsilon^{-c(l)} $ times the $ l $-th Sobolev norm of $ \chi $, where $ c(l) $ is some universal constant depending on $ l $.
 Define the thickening of $ w $ by
 \begin{align}
 w_\epsilon(k,x) := w(k) \cdot \chi_\epsilon(k,\epsilon).
 \end{align}
 Note that this function satisfies by construction $ \int w_\epsilon(k,x)dx = w(k) $.
 The weaker identity required by \autoref{eqn:w_eps_coset_average} follows by integrating this along $ K_v $-cosets.
\end{proof}
% Next, on the corresponding subspaces choose cutoff functions $\chi_{\frakp,\epsilon},\chi_{\frakn,\epsilon},\chi_{\fraka,{\epsilon}}$ with the following properties:
% \begin{enumerate}
%  \item The value at the origin is $1$, the functions are positive, and their support is in the $\epsilon$-neighborhood of the origin.
%  \item The total integral of each function is $1$, and they satisfy the Sobolev norm bound of the type \eqref{eqn:sobolev_bd_cpct_thick}.
% \end{enumerate}
% Such functions are constructed by appropriate rescalings of some fixed smooth bump functions.
% Define now the desired function, using the coordinates from \eqref{eqn:k_thick_coords}, by:
% \begin{align}
%  w_\epsilon := w(k) \cdot \chi_{\frakp,\epsilon}(p)\cdot \chi_{\frakn,\epsilon}(n) \cdot \chi_{\fraka,{\epsilon}}(a)
% \end{align}
% It is clear that the first two properties in the statement of the proposition hold.
% To see the last property, note that the family of orbits $kH_v a_t$ has coordinates (with respect to \eqref{eqn:k_thick_coords}) in the $K$-factor lying in the coset $kK_v$.
% The result follows by performing the integration.

\subsection{The count}
\label{ssection:the_count}
This section assembles the ingredients discussed above and proves the counting from \autoref{thm:homogeneous_counting}, as formulated in \autoref{eqn:homogeneous_asymptotic_exp}.
Let $ w $ be the function appearing in the statement, viewed as a $K_v$-invariant function on the compact group $ K $.
\autoref{prop:w_epsilon} gives a family of functions $ w_\epsilon $ to which the arguments below apply.

\subsubsection{Quantitative Equidistribution}

Denote by $w_\epsilon$ the function constructed in \autoref{prop:w_epsilon} but projected to $\leftquot{\Gamma}{G}$.
To define it, fix $\epsilon_0$ less than the injectivity radius of $\Gamma\cdot K \subset \leftquot{\Gamma} {G}$ and assume that $\epsilon <\epsilon_0$ throughout.
\autoref{thm:quantitative_equidist} in the next section gives quantitative equidistribution of the right $ a_t $-translates of $ Y $, namely universal constants $l,\delta>0$ such that
\begin{align}
\label{eqn:equidistr_prob}
\int_{Y a_t}^{} w_\epsilon \, d\mu_{Y a_t} = \frac{\Vol Y}{\Vol X} \left(\int_X w_\epsilon \, d\mu_X\right) + O\left(\norm{w_\epsilon}_l e^{-\delta \cdot t}\right)
\end{align}
where the implied constant in $O(-)$ is absolute.
Recall that $ d\mu_{Y a_t} $ denoted a fixed volume measure on $ Y a_t $ and $ \Vol Y $ is taken for this measure.
The factor $ \norm{w_\epsilon}_l $ denotes the $ l $-th Sobolev norm of $ w_\epsilon $.

\subsubsection{Approximate Counting Function}
Changing the measure in \autoref{eqn:equidistr_prob} to $ d\mu_{Y_t} $ (see \autoref{prop:msr_scaling}) simply multiplies both sides by a factor of $ e^{20 \cdot t} $:
\begin{multline}
\label{eqn:ft_asymptotics}
f(t,\epsilon) := \int_{Y a_t}^{} w_\epsilon \, d\mu_{Y_t}= \\ =\frac{\Vol Y}{\Vol X} \left(\int_X w_\epsilon \, d\mu_X\right)e^{20 \cdot t} + O\left(\norm{w_\epsilon}_l e^{(20-\delta) \cdot t}\right).
\end{multline}
The approximate counting function is defined to be the integrated version of this quantity:
\begin{align}
\label{eqn:FT_def}
F(T,\epsilon) &:= \int_{-\infty}^{T} f(t,\epsilon)\, dt  = \int_{-\infty}^{T} \left(\int_{Y a_t}^{} w_\epsilon \, d\mu_{Y_t} \right)\, dt.
\end{align}
Integrating the estimate from \autoref{eqn:ft_asymptotics} gives:
\begin{align}
\label{eqn:FT_asymptotics}
\begin{split}
F(T,\epsilon) & = \frac{\Vol Y}{20\cdot \Vol X} \left(\int_X w_\epsilon \, d\mu_X\right)e^{20 \cdot T} + O\left(\norm{w_\epsilon}_l e^{(20-\delta) \cdot T}\right)\\
& = \frac{\Vol Y}{20\cdot \Vol X} \left(\int_K w \, d\mu_K\right)e^{20 \cdot T} + O\left(\norm{w_\epsilon}_l e^{(20-\delta) \cdot T}\right).
\end{split}
\end{align}

\subsubsection{Counting Function}
\label{sssec:CT}
We also have the counting function we are interested in:
\begin{align*}
\begin{split}
C(T) & := \sum_{[\gamma]\colon t([\gamma])\leq T}^{} w(C_{[\gamma]}) \\
% & = \sum_{[\gamma]\colon t([\gamma])\leq T}^{} \int_{t([\gamma])-\epsilon}^{t([\gamma])+\epsilon} \left( \int_{Y a_t} \! w_\epsilon \, d\mu_{Y_t} \right)dt
\end{split}
\end{align*}
% where the second equality follows from \autoref{prop:w_epsilon}, specifically the integral comparison in \autoref{eqn:w_eps_coset_average}.
which can be controlled by $F(T,\epsilon)$ as follows:

\begin{proposition}[Comparison]
	\label{prop:comparison}
	We have the inequalities:
	\begin{align*}
	C(T-\epsilon) \leq F(T,\epsilon) \leq C(T+\epsilon)
	\end{align*}
  for all $T\in \bR$ and $0<\epsilon<\epsilon_0$.
\end{proposition}
Note that $C(T)=0$ for $T$ sufficiently negative, and the same is true for $F(T,\epsilon)$ once $\epsilon$ is sufficiently small.
\begin{proof}
	By construction $ w_\epsilon $ is supported near $\Gamma\cdot  K $ in $X=\leftquot{\Gamma}{G}$.
  Let $\tilde{w}_\epsilon$ denote the pullback of $w_\epsilon$ to $X_v=\leftquot{\Gamma_v}{G}$.
  Then $\tilde{w}_\epsilon$ is the sum of functions supported in the $\epsilon$-neighbrhood of the $K$-orbits $ \leftquot{\Gamma_v}{\Gamma}\cdot K $, and with disjoint supports.
	But \autoref{prop:inters_biject} gives a bijection between intersections of $ \leftquot{\Gamma_v}{\Gamma}\cdot K $ with $ a_t $-translates of $ Y $ in $X_v$ and cosets $ [\gamma] \in \rightquot{\Gamma}{\Gamma_v}$.
  So we can write
  \[
    \tilde{w}_\epsilon = \sum_{[\gamma]} \tilde{w}_{\epsilon , [\gamma]}
  \]
  and from \autoref{eqn:w_eps_coset_average} it follows that
  \[
    C(T)= \sum_{[\gamma]\colon t([\gamma])\leq T} \int_{-\infty}^{+\infty} \left(\int_{Y a_t}^{} \tilde{w}_{\epsilon,[\gamma]} \, d\mu_{Y_t} \right)\, dt
  \]
  where the contribution in the integral comes only for the values of $t$ satisfying $|t([\gamma])-t|\leq \epsilon $.

  The inequalities in the proposition then follow because each of the quantities can be expressed as integrals of functions that obey the desired inequality pointwise.
\end{proof}

Below, the notation $A(T) = B(T\pm \epsilon)$ means $B(T-\epsilon)\leq A(T) \leq B(T+\epsilon)$.

\subsubsection{Conclusion}
\label{sssec:conclusion}
\autoref{prop:comparison} and the asymptotic expansion of $ F(T,\epsilon) $ from \autoref{eqn:FT_asymptotics} give
\begin{align*}
C(T) &= \frac{\Vol Y}{20 \Vol X} \left(\int_K w \, d\mu_K\right)e^{20\cdot (T\pm \epsilon)} + O\left(\norm{w_\epsilon}_l \cdot e^{(20-\delta)T} \right)
\intertext{ and using the Sobolev norm bound on $ w_\epsilon $ from \autoref{eqn:sobolev_bd_cpct_thick}:}
C(T) &= \frac{\Vol Y}{20 \Vol X} \left(\int_K w\, d\mu_K\right)e^{20\cdot (T\pm \epsilon)} + O(\epsilon^{-d_l} \cdot e^{(20-\delta)T} ).
\end{align*}
Set now $ \epsilon = e^{-\delta_1 \cdot T} $ and choose $ \delta_1 $ sufficiently small, e.g. $ \delta_1:= \frac{\delta}{2(d_l+1)} $ to obtain the final estimate (using $ e^{-A\cdot \epsilon} =  1 - A\epsilon + O(\epsilon^2) $):
\begin{align}
C(T) = \frac{\Vol Y}{20 \Vol X} \left(\int_K w \, d\mu_K\right)e^{20\cdot T} + O(e^{(20-\delta_1)T} ).
\end{align}
This finishes the proof of \autoref{thm:homogeneous_counting}. \hfill \qed

\section{Equidistribution}
\label{sec:equidistribution}

This section establishes the quantitative equdistribution of right $ Y $-translates by $ a_t $.
The strategy is to replace $ Y $ by a bump function supported near it and use general quantitative mixing results for functions.
The wavefront property allows one to compare integrals on translates of $ Y $ and integrals against the translated bump function.
Because $ Y $ is non-compact, to construct the bump function we need to cut off the part of $ Y $ that lies sufficiently deep in the cusp.

Throughout, constructions will depend on a parameter $ \epsilon>0 $ which will eventually be $ e^{-\delta \cdot t} $ for some $ \delta>0 $.
However, in most arguments we can freely replace $ \epsilon $ by a fixed positive power of $ \epsilon $.
This refers to injectivity radii (if we consider $ \epsilon $ or $ \epsilon^{1/2} $ is irrelevant) as well as the Sobolev norms of various functions.

\subsection{Cutoffs, Thickenings, and Wavefronts}

In this section, following Benoist--Oh \cite{Benoist_Oh} we construction a function $ \phi_\epsilon $ on $ X $ which will serve as a replacement for $ Y $.
For this, we first need to cut off the part of $ Y $ with small injectivity radius (living in the cusp of $ X $).
Next, there are natural transverse directions to $ Y $ and $ \phi_\epsilon $ is constructed using bump function in those transverse directions.
Finally, we establish the analogue of the wavefront lemma of Eskin--McMullen \cite{EskinMcMullen} in this setting.

\begin{figure}[ht]
	\centering
	\includegraphics[width=\textwidth]{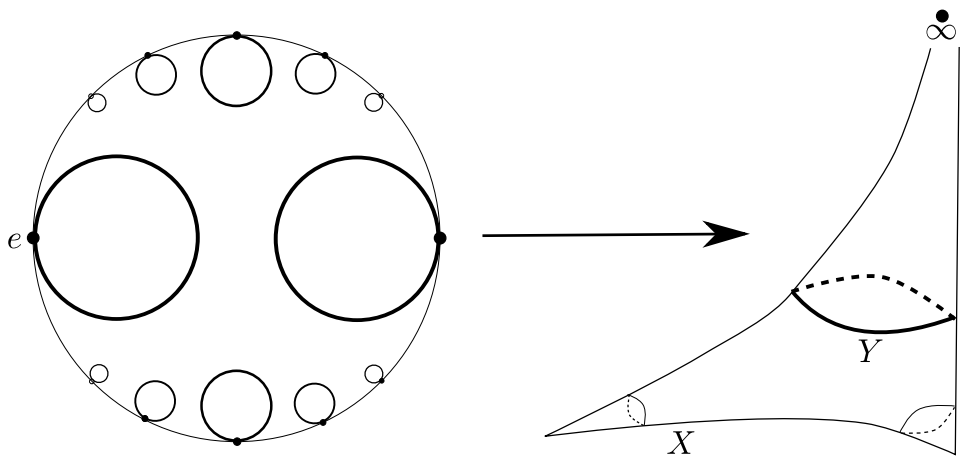}
	\caption{The equidistribution of $ Y $ established in this section is analogous to the equidistribution of a closed horocycle on the modular surface, under the geodesic flow that expands it.}
	\label{fig:horocycles}
\end{figure}

\subsubsection{Injectivity radius}
Let $ B_r \subset G$ denote the ball of radius $ r $ in $ G $, centered at the identity.
For any point $ x\in X $ define
\begin{align*}
\inj(x) := \sup_{r\geq 0} \lbrace B_r \to X\colon  g \mapsto x\cdot g \quad \textrm{is injective}\rbrace
\end{align*}
and set
\begin{align*}
Y_\epsilon := \{y\in Y \colon \inj(y)\geq \epsilon\}.
\end{align*}

Following Benoist--Oh \cite{Benoist_Oh}, the volume of points in $ Y $ with small injectivity radius is bounded by a power of $ \epsilon $.
\begin{proposition}
	\label{prop:inj_radius_volume_bound}
	There exist constants $ C, p >0 $ such that
	\begin{align*}
	\mu_Y (Y\setminus Y_\epsilon) \leq C \epsilon^p.
	\end{align*}
\end{proposition}
\begin{proof}
	This claim is essentially the same as that of \cite[Lemma 11.2]{Benoist_Oh}, except that $ H_v $, the group locally modeling $ Y $, is not affine symmetric but rather a semidirect product of a semisimple group and a unipotent part (see \autoref{sssec:He_dec}).
	We only recall the main steps in the proof.

	Reduction theory \cite[Ch. 4]{Platonov_Rapinchuk} applies in the same way, except that the unipotent part is larger.
	One has a decomposition $ H_v = N_v A_v K_v $ and a Siegel set $ \Sigma_{\omega,t_0} = \omega \cdot A_{v, t_0}\cdot K_v $ where $ \omega\subset N_v $ is a compact set and $ A_{v,t_0} $ is a cone in a Cartan subalgebra of $ H_v $.
	It has the property that it is a coarse fundamental domain for the action of $ \Gamma_v $, namely there exist finitely many $ h_i \in H $ such that $ H_v = \cup_i \Gamma_v \cdot h_i \Sigma_{\omega, t_0} $.

	Next, Benoist--Oh choose a cutoff $A_\epsilon \subset  A_{v, t_0} $ and show that the injectivity radius on the image of $ \Gamma_v \cdot h_i \cdot \omega  A_\epsilon K_v $ in $ X $ is bounded below by $ \epsilon $.
	Additionally, they check that the volume of the complement is bounded by a positive power of $ \epsilon $.
	This finishes the proof.
\end{proof}

\subsubsection{Weak Stable Direction}
From the decomposition of the Lie algebra $ \frakg $ in \autoref{sssec:weight_dec}, define
\begin{align*}
\frakn' := \frakn^- \oplus \fraka \textrm{ with associated Lie group } N':=\exp(\frakn').
\end{align*}
Note that $ \frakg = \frakh_v\oplus \frakn' $ where $ \frakh_v $ is the Lie algebra of $ H_v=\Stab_G v $, and corresponds to the tangent bundle of $ Y $.
Under the adjoint action of $ a_t $, the algebra $ \frakn' $ is weakly contracted: $ \frakn $ is strongly contracted, and $ \fraka $ is fixed.

Throughout, a subscript of $ \epsilon $ in a Lie algebra (resp. group) denotes an $ \epsilon $-neighborhood of the origin (resp. identity).
\begin{proposition}[Local Product Structure]
	\label{prop:loc_prod_struct}
	For all sufficiently small $ \epsilon $, the map
	\begin{align*}
	Y_{\epsilon^{1/2}} \times N'_\epsilon & \to X\\
	(y,n') &\mapsto y\cdot n'
	\end{align*}
	is injective.
\end{proposition}
\begin{proof}
	Recall that $ Y $ is locally modeled on the subgroup $ H $.
	Because of the Lie algebra decomposition of $ G $ from \autoref{sssec:weight_dec}, for sufficiently small $ \epsilon $ the map $ H_\epsilon\times N'_\epsilon \to G$ is injective.
	Now for any $ y\in Y_{\epsilon^{1/2}} $ the map $ y\cdot G_\epsilon \to X $ is injective (by the injectivity radius assumption, and since $ \epsilon^{1/2}\gg 2\epsilon $).
	The claim now follows.
\end{proof}

To avoid cluttering notation, from now on we assume that $ Y_\epsilon $ denotes the subset of $ Y $ with injectivity radius at least $ \epsilon^{1/2} $, so that the above proposition applies.

\begin{proposition}[Thickening $ Y $]
	\label{prop:Y_thickening}
	For all sufficiently small $\epsilon>0$
	%(i.e. less than $\epsilon_0$ from \autoref{prop:existence_transverse})
	there exist smooth functions $\tau_\epsilon:Y_\epsilon \to \bR_{\geq 0}$  and $ \rho_\epsilon\colon N'_\epsilon \to \bR_{\geq 0} $ with the following properties:
	\begin{enumerate}
		\item The support of $\tau_\epsilon$ is in $Y_\epsilon$, it satisfies $ \tau_\epsilon \leq 1 $ everywhere and $ \tau_\epsilon\vert_{Y_{4\epsilon}} \equiv 1 $.
		\item The size of $ \rho_\epsilon $ is normalized to:
		\begin{align*}
		\int_{N'_\epsilon} \rho_\epsilon = 1.
		\end{align*}
		\item The function defined (see \autoref{prop:loc_prod_struct})  for $ x = yn' \in Y_\epsilon\times N'_\epsilon $ as
		\begin{align}
		\label{eqn:phi_epsilon_def}
		\phi_\epsilon (yn') := \tau_\epsilon(y)\cdot \rho_\epsilon(n')
		\end{align}
		satisfies the following Sobolev norm bound, for every $l\in \bN$ and corresponding constant $C_l$:
		\begin{align}
		\label{eqn:phi_epsilon_sobolev_bound}
		\norm{\phi_\epsilon}_l \lesssim_l \epsilon^{-C_l}.
		\end{align}
		In fact, $C_l$ depends linearly on $l$.
	\end{enumerate}
\end{proposition}
\begin{proof}
	Again, the claim is almost identical to \cite[Prop. 11.7]{Benoist_Oh}.
	The difference is that their $ W_\epsilon $ is our $ N'_\epsilon $, and their $ H_S $ is our $ H_v $, but $ H_S $ is affine symmetric whereas $ H_v $ is a semidirect product of a semisimple part and a unipotent part.
	This, however, does not affect the proof.
	We only recall the main steps.

	The first step, \cite[Lemma 11.8]{Benoist_Oh}, is to build local bump functions $ \beta_\epsilon $ on $ H $ which are supported in a neighborhood of the origin $ H_\epsilon $ and have the expected Sobolev norm bounds.
	Additionally, $ \beta_\epsilon \geq 1$ on $ H_{\epsilon^2} $.
	This construction is by rescaling a fixed smooth bump function.

	Next, one chooses a maximal collection of points $ y_i\in Y_\epsilon $ such that the balls $ B(y_i,\epsilon^3)\subset Y $ are disjoint.
	Since the volume of $ Y_\epsilon $ is bounded by $ 1 $, there are at most $O(\epsilon^{-3\dim Y} )$ such points.
	Next, one places copies of the bump functions $ \beta_\epsilon $ centered at the $ y_i $, denoted $ \beta_{y_i,\epsilon} $
	The desired cutoff function is then defined by
	\begin{align*}
	\tau_\epsilon := \frac{\sum_{y_j \in \cF} \beta_{y_j,\epsilon} }{\sum_{y_i \in \cG} \beta_{y_i,\epsilon} }
	\end{align*}
	where $ \cF, \cG $ denote the collection of $ y_i$ which are contained in $ Y_{4\epsilon} $ and $ Y_{2\epsilon} $ respectively.

	The function $ \rho_\epsilon $ is defined as a rescaling of a fixed bump function on $ N' $.
	The estimates for the Sobolev norms are checked directly.
\end{proof}

A key result introduced in \cite{EskinMcMullen} is the wavefront lemma.
In the adaptation needed for our result, it says that points near $ Y $, when translated by $ a_t $, stay near the $ a_t $-translate of $ Y $.
We will state it below for comparisons of integrals of translates of functions.

Since $ G $ acts on $ X $ on the right, define the translation action of $ G $ on functions $ w:X\to \bR $ by
\begin{align}
\label{eqn:G_action_fct}
(w\cdot g)(x) := w(x\cdot g^{-1}).
\end{align}

\begin{lemma}[Wavefront Lemma]
	\label{lem:wavefront}
	Let $ w:X\to \bR $ be a smooth compactly supported function.
	Let $ \phi_\epsilon $ be the function constructed in \autoref{prop:Y_thickening}.
	Then
	\begin{multline*}
	\int_X w \cdot (\phi_\epsilon \cdot a_t) \, d\mu_X
	= \int_Y w(ya_t) d\mu_{Y}(y) + \\
	+ O(\epsilon\cdot \Lip(w)) + O(\epsilon^{p_1}\cdot \norm{w}_{L^\infty})
	\end{multline*}
	where $ \Lip(w) $ denotes the Lipschitz constant of $ w $ and $ \norm{w}_{L^\infty} $ denotes its supremum.
\end{lemma}
\begin{proof}
Recall that by construction, the support of $ \phi_\epsilon $ is in $ Y_\epsilon\times \frakn'_\epsilon $, and therefore the support of $ \phi_\epsilon \cdot a_t $ is in the translation on the right by $ a_t $ of this set (see \autoref{eqn:G_action_fct}).

We have for $ x = y n a \cdot a_t $ that
\begin{align*}
(\phi_\epsilon\cdot a_t) (x ) = \phi_\epsilon(y n a a_t \cdot a_{-t}) = \phi_\epsilon(yna) = \tau_\epsilon(y)\cdot \rho_\epsilon(na)
\end{align*}
where we have used the definition of $ \phi_\epsilon $ from \autoref{eqn:phi_epsilon_def}.
Similarly, for $ w(x) $ we have
\begin{align*}
w(x) = w(y\cdot n a \cdot a_t) = w(y a_t \cdot a_{-t}na_t\cdot a ) = w(y a_t) + O(\epsilon \cdot \Lip(w))
\end{align*}
where $ \Lip(w) $ denotes the Lipschitz norm of $ w $, and we used the exponential contraction of the $ a_t $-action on $ \frakn^- $ and that $ \norm{a} =  O(\epsilon)$.

We can now estimate the integral of interest as follows:
\begin{multline}
\int_X w\cdot (\phi_\epsilon \cdot a_t) d\mu_X = \int\displaylimits_{(Y_\epsilon\times \frakn'_\epsilon)\cdot a_t}^{} w(x) (\phi_\epsilon \cdot a_t)(x) d\mu_X(x) \\
= \int_{Y_\epsilon} \left[\int_{\frakn'_\epsilon} w(ya_t) \tau_\epsilon(y)\rho_\epsilon(na) dn \, da \right] d\mu_{Y} + O(\epsilon\cdot \Lip(w))\\
= \int_Y w(ya_t) d\mu_{Y} + O(\epsilon\cdot \Lip(w)) + O(\epsilon^{p_1}\cdot \norm{w}_{L^\infty})
\end{multline}
where for the last line we used the estimate for the volume of $ Y\setminus Y_\epsilon $ from \autoref{prop:inj_radius_volume_bound}.

For sufficiently large $ l $, the Sobolev norm $ \norm{w}_l $ controls both $ \Lip(w) $ and $ \norm{w}_{L^\infty} $ (since $ w $ has compact support).
Therefore, the desired estimate follows.
\end{proof}

\subsection{Spectral theory}
\label{ssec:spectral_th}

This section recalls the definition of Sobolev spaces for non-compact quotients of semisimple Lie groups.
It also contains the necessary quantitative mixing result from the spectral theory of semisimple Lie groups.

\subsubsection{Setup}
Let $ G $ be a semisimple real Lie group with a lattice $ \Gamma\subseteq G $.
Denote by $ \Lie G $ the Lie algebra of $ G $, canonically identified with the left-invariant vector fields on $ G $.
Then the universal enveloping algebra $ \cU(\Lie G) $ can be identified with the left-invariant differential operators on $ G $.

Consider the quotient $ X:=\leftquot{\Gamma}{G} $ equipped with the Haar measure $ \mu_X $.
Because $ \cU(\Lie G) $ denotes left-invariant operators on $ G $, it descends to an algebra of differential operators on $ X $, denoted $ \cU $.

\subsubsection{Sobolev spaces}
\label{sssec:sobolev_spaces}
Fix a basis $ D_i $ of $ \Lie G $ and using the Poincar\'{e}--Birkhoff--Witt theorem, a basis $ D_{i_1,\ldots, i_k} $ of the universal enveloping algebra $ \cU(\Lie G) $.
For a smooth function $ \alpha \in C^\infty (X) $ define the $ l $-th order $ L^2 $ Sobolev norm by
\begin{align}
\label{eqn:sobolev_def}
	\norm{\alpha}_{l}^2 := \sum_{0\leq k\leq l} \norm{D_{i_1\ldots i_k}\alpha}_{L^2}^2.
\end{align}
Here $ \norm{-}_{L^2} $ denotes the $ L^2 $ norm on the space $ L^2(X,\mu_X) $, and the summation is over a basis of operators of order $ \leq l $.

\begin{remark}
	\label{rmk:sobolev_casimir}
	There is an alternative way to define a Sobolev norm, which is comparable to the one in \autoref{eqn:sobolev_def} up to constants which only depend on the Lie group.
	For this, fix a maximal compact $ K\subset G $ and an associated Casimir operator $ \Omega \in \cU(\Lie G) $ which can be viewed as a Laplacian.
	Then up to multiplicative constants we have
	\begin{align}
	\norm{\alpha}_{{2l}}^2 \approx \norm{\alpha}_{L^2}^2 + \norm{\Omega^l \alpha}^2_{L^2}.
  \end{align}
\end{remark}

\subsubsection{Quantitative mixing}
\label{sssec:quant_mixing}
The right action of $G$ on functions is defined, for $g\in G$ and $\alpha$ a function on $X$, using right push-forward.
The formula becomes $(\alpha\cdot g) (x) := \alpha(x\cdot g^{-1})$.

The next result is established, for instance, in \cite[Cor. 3.5]{Kleinbock_Margulis}.
It applies in our context with $ G $ and $ \Gamma $ as above (note that the center of $ G $ is finite and contained in $ \Gamma $, so we can just quotient by the center from the start).
\begin{theorem}
\label{thm:quantitative_mixing}
	Let \(G\) be a connected semisimple center-free Lie group without compact factors, and \(\Gamma\subset G\) an irreducible non-uniform lattice.
	Set \(X:=\leftquot{\Gamma}{G} \) with normalized Haar probability measure $\mu$.
	There exist constants $ A, l, \delta >0$ depending only on $ G $ and $ \Gamma $ such that for all $ \alpha, \beta \in C^2_\infty (X) $ and $ g\in G $ we have
	\begin{align}
	\norm{ \left(\int_X \! \alpha \cdot \left(\beta\cdot g\right)  \, d\mu \right) - \int_X \alpha \cdot \beta\, d\mu }
	\leq A \cdot \norm{\alpha}_l \cdot \norm {\beta}_l \cdot e^{-\delta \cdot \norm {g} }
	\end{align}
	Here $ C^2_\infty(X) $ denotes the space of smooth functions $ \alpha $ on $ X $ with $ D\alpha \in L^2(X,\mu) $, for $ D $ any differential operator in the universal enveloping algebra of $ G $.
	Therefore a function $\alpha \in C^2_\infty(X) $ is in the intersection of all Sobolev spaces on $ X $.
	The Sobolev norms are denoted $ \norm{-}_l $.
\end{theorem}

The key point of the above theorem is that the constants $ A, l $ as well as the spectral gap $ \delta $ do not depend on the functions $ \alpha, \beta $.

\subsection{Quantitative Equidistribution}

We keep the notation from the previous two sections.
The next result establishes in a quantitative form the needed equidistribution of $ Y $-translates under the right action of $ a_t $.

\begin{theorem}
	\label{thm:quantitative_equidist}
	There exist constants $l,\delta>0$ such that for all $w\colon X\to \bR$ smooth, compactly supported, we have as $t\to \infty$:
 \begin{align*}
  \int\displaylimits_{Y \cdot a_t}\! w \, d\mu_{Y a_t} = \frac{\Vol Y}{\Vol X}\cdot \left(\int\displaylimits_{X}\! w\, d\mu_X\right) + O\left(\norm{w}_l \cdot e^{-t\cdot \delta}\right).
 \end{align*}
 Here $\norm{w}_l$ denotes the $l$-th order Sobolev norm, and the implied constants in $O(-)$ are absolute.
 The measure $d\mu_{Ya_t}$ has volume independent of $t$ and is defined in \autoref{sssec:cont_vol_msr}.
\end{theorem}

\begin{proof}
 Consider the function $\phi_\epsilon$ constructed in \autoref{prop:Y_thickening}, which serves as a thickening of $Y$.
 \autoref{thm:quantitative_mixing} gives $\delta'>0$ such that
 \begin{multline}
 \label{eqn:spectral_approx}
  \int_{X}\! w \cdot (\phi_\epsilon \cdot a_t)\, d\mu_X  = \\
  = \frac 1 {\Vol X} \left(\int_X \! w\, d\mu_X\right) \left(\int_X \! \phi_\epsilon\, d\mu_X \right) + O\left(\norm{w}_l \cdot \norm{\phi_\epsilon}_l \cdot e^{-t\cdot \delta'}\right)\\
  = \frac {\Vol Y} {\Vol X} \left(\int_X \! w\, d\mu_X\right) + O\left(\norm{w}_l \cdot \norm{\phi_\epsilon}_l \cdot e^{-t\cdot \delta'}\right)
 \end{multline}
 where we used that
 \begin{align*}
 \int_X \phi_\epsilon \, d \mu_X = \Vol Y + O(\Vol (Y\setminus Y_\epsilon)) = \Vol Y + O(\epsilon^{p_1})
 \end{align*}
 which follows from \autoref{prop:inj_radius_volume_bound}.
 Now \autoref{eqn:phi_epsilon_sobolev_bound} gives $ \norm{\phi_\epsilon}_l = \epsilon^{-c_l} $.
 Set $ \epsilon = e^{-t\cdot \delta''} $ for a sufficiently small $ \delta''>0 $.
 The claim will follow after we show that
 \begin{align*}
 \int_X w \cdot (\phi_\epsilon \cdot a_t) \, d\mu_X = \int_{Ya_t}^{} \! w \, d\mu_{Ya_t} + O\left(\norm{w}_l\cdot \epsilon^{c_1}\right)
 \end{align*}
 for some exponent $ c_1>0 $.

 But this last comparison follows from the \hyperref[lem:wavefront]{Wavefront Lem\-ma \ref{lem:wavefront}}.
 Indeed, both the Lipschitz norm $ \Lip(w) $ and the supremum $ \norm{w}_{L^\infty} $ are controlled by a sufficiently high Sobolev norm, since $ w $ is compactly supported.
\end{proof}

\subsubsection{Equidistribution of poles and equators}
\label{sssec:Funk_transform}
The result stated in \autoref{thm:main_slag_counting} of the introduction refers to special Lagrangian fibrations, and therefore to the equators with poles corresponding to elliptic fibrations.
While the count is the same up to a factor of $ 1/2 $, equidistribution of equators follows from equidistribution of their corresponding poles.
Let us see why.

Recall that the Funk transform $ \cF\colon L^2(\bS^2)\to L^2(\bS^2) $ takes a function to its integral over corresponding equators.
The image consists of all even functions, and the kernel are the odd functions (for the antipodal involution).
Because the transform commutes with orthogonal transformations, it acts by scalars on the space of harmonic polynomials.
Namely, the Funk transform on harmonic polynomials annihilates the odd degree ones, and acts in degree $ 2n $ by the scalar $ (-1)^n \frac{1\cdot 3\cdots (2n-1)}{2\cdot 4\cdots (2n)} $ (\cite{Funk}).
From this, (quantitative) equidistribution of equators becomes equivalent to (quantitative) equidistribution of their poles.

\bibliographystyle{sfilip}
\bibliography{counting_sLags.bib}
%====================================================
\end{document}